\documentclass[a4paper]{scrartcl}
\usepackage[utf8]{inputenc}
\usepackage[english]{babel}
\usepackage{xcolor}
\usepackage{graphicx}
\usepackage{hyperref}
\usepackage[T1]{fontenc}
\usepackage{lmodern}
\usepackage{rotating}
\usepackage{tikz}
\usepackage{tikz-cd}
\usepackage[backend=biber,url=false,doi=false,isbn=false,giveninits=true,style=alphabetic,maxbibnames=1000,date=year,sorting=nyt]{biblatex}
\bibliography{references.bib}
\DeclareFieldFormat*{title}{\mkbibemph{#1\isdot}}
\DeclareFieldFormat*{journaltitle}{#1\isdot}
\renewbibmacro{in:}{}
\usepackage{amsmath}
\usepackage{amsfonts}
\usepackage{amssymb}
\usepackage{amsthm}
\usepackage{mathtools}
\usepackage{stmaryrd}
\usepackage{enumerate}
\usepackage{verbatim}
\usepackage{dsfont}
\usepackage{mathabx}
\usepackage{faktor}
\usepackage[normalem]{ulem}
\DeclareMathOperator\defect{def}
\DeclareMathOperator\Gal{Gal}

\DeclareMathOperator\wt{wt}

\DeclareMathOperator\avg{avg}
\DeclareMathOperator\supp{supp}

\DeclareMathOperator\LP{LP}
\DeclareMathOperator\cl{cl}

\DeclareMathOperator\Stab{Stab}
\DeclareMathOperator\rk{rk}

\DeclareMathOperator\id{id}
\DeclareMathOperator\conv{conv}
\def\GL{{\mathrm{GL}}}

\def\dom{{\mathrm{dom}}}

\hypersetup{colorlinks=false, linkbordercolor={white}, hidelinks, pdfauthor={Felix Schremmer, Ryosuke Shimada and Qingchao Yu},pdftitle={Affine Deligne-Lusztig Varieties of Positive Coxeter Type}}
\usepackage{csquotes}
\author{\textsc{Felix Schremmer, Ryosuke Shimada and Qingchao Yu}}%\date{\today}
\title{Affine Deligne-Lusztig Varieties of Positive Coxeter Type}
\numberwithin{equation}{section}
\newtheorem{theorem}[equation]{Theorem}
\newtheorem{alphatheorem}{Theorem}
\newtheorem{proposition}[equation]{Proposition}
\newtheorem{alphaproposition}[alphatheorem]{Proposition}
\newtheorem{lemma}[equation]{Lemma}
\newtheorem{corollary}[equation]{Corollary}

\theoremstyle{definition}
\newtheorem{definition}[equation]{Definition}

\theoremstyle{remark}
\newtheorem{example}[equation]{Example}
\newtheorem{remark}[equation]{Remark}

\def\abs#1{{\left\lvert{#1}\right\rvert}}
\def\doubleparen#1{{(\!({#1})\!)}}

\let\oldqedsymbol\qedsymbol
\def\qedaddendum{}
\def\qedsymbol{\oldqedsymbol\qedaddendum}

\def\af{{\mathrm{af}}}
\def\rightqed{\pushQED{\qed}\qedhere\popQED}

\def\a{\alpha}
\def\b{\beta}

\def\G{\Gamma}

\def\D{\Delta}

\def\s{\sigma}
\def\t{\tau}

\def\k{\kappa}
\def\l{\lambda}

\def\<{\langle}
\def\>{\rangle}
\def\bq{\mathbf q}

\newcommand{\BA}{\mathbb {A}}

\newcommand{\BF}{\mathbb {F}}
\newcommand{\BG}{\mathbb {G}}

\newcommand{\BJ}{\mathbb {J}}

\newcommand{\BN}{\mathbb {N}}

\newcommand{\BP}{\mathbb {P}}
\newcommand{\BQ}{\mathbb {Q}}
\newcommand{\BR}{\mathbb {R}}
\newcommand{\BS}{\mathbb {S}}

\newcommand{\BZ}{\mathbb {Z}}

\newcommand{\CI}{\mathcal {I}}

\newcommand{\CO}{\mathcal {O}}
\newcommand{\CP}{\mathcal {P}}

\newcommand{\CT}{\mathcal {T}}

\def\tW{\widetilde W}

\newcommand{\Ad}{{\mathrm{Ad}}}

\newcommand{\de}{{\mathrm{def}}}

\newcommand{\remind}[1]{{\textbf {** #1 **}}}

\renewcommand{\setminus}{-}

\begin{document}
\maketitle
%\tableofcontents
\begin{comment}

\textbf{TODO.}
\begin{itemize}
\item \sout{Explain the connection to Shimura varieties better (Ryosuke's articles, and mention that finite Coxeter type is not enough)}
\item \sout{Explain how much simpler the proof is compared to \cite{He2022}.}
\item Add a result proving the product structure (splitting of fibres) for $G$ split and $b$ basic. Can we generalize \cite[Theorem~11.3.1]{Goertz2010} and Falting's result to arbitrary $b$ and quasi-split groups? I.e.\ we want
\begin{align*}
\dim X_x(b)\cap IyI = \dim ( IxI/I ~\cap~ y^{-1} I b\sigma(y) I/I) -\langle \nu + \nu^{\dom},\rho\rangle
\end{align*}
if $b$ is straight and $\nu$ its non-dominant Newton point. -- done for $[b]$ basic.
\item Do we need Falting's paper or should we remove the reference?
\item \sout{Can we find a title that is more catchy to Shimura varieties people?}
\item Which editor to contact at Cambridge Math J? Felix would suggest Mark Kisin.
\item Cleanup the Introduction.
\end{itemize}
\end{comment}

\begin{abstract}
We introduce a class of affine Deligne--Lusztig varieties that we call of positive Coxeter type. We show that the affine Deligne--Lusztig varieties of positive Coxeter type have a very simple and explicitly described geometric structure. Conversely, we explain how some of these geometric properties can be used to characterize this class.
These results vastly generalize the work of He--Nie--Yu on affine Deligne–Lusztig varieties of finite Coxeter type, leading to applications to Shimura varieties that were not possible using the old notion.
\end{abstract}

\section{Introduction}
\subsection{Background}
\begin{comment}
Each Coxeter group comes with a set of canonical generators, called its \emph{simple reflections}. By multiplying these simple reflections in any order, one obtains the so-called \emph{Coxeter elements}. These Coxeter elements have been studied with great interest, and play a crucial role in the structure theory of Coxeter groups.

In a seminar article \cite{Deligne1976}, Deligne and Lusztig introduce a certain class of varieties whose cohomology can be used to generate many characters of finite groups of Lie type. The geometric and cohomological properties of these Deligne--Lusztig varieties have been studied with great interest in the past. Those Deligne--Lusztig varieties which are associated to Coxeter elements in the Weyl group are especially well--understood and have many special properties \cite{Lusztig1976}. In our article, we introduce and study a class of \emph{affine} Deligne--Lusztig varieties with similar properties.
    
\end{comment}

The notion of affine Deligne--Lusztig varieties are introduced by Rapoport, cf. \cite{Rapoport2002}. These are schemes associated with any reductive group $G$ defined over a local field $F$, and arise naturally when studying the reduction of Shimura varieties (in the case of mixed characteristic) or moduli spaces of local $G$-shtukas (in the case of equal characteristic).

Broadly speaking, one may summarize that the affine Deligne--Lusztig varieties compare two natural stratifications of the group $G(\breve F)$, where $\breve F$ is the completion of the maximal unramified extension of $F$.

Denoting the Frobenius of $\breve F/F$ by $\sigma$. The first stratification of $G(\breve F)$ is given by the $\sigma$-conjugacy classes. Let
\begin{align*}
B(G) &= \{[b]\mid b \in G(\breve F)\},
\end{align*}
where $[b]=\{g^{-1} b \sigma(g)\mid g\in G(\breve F)\}$. This set has been studied and used quite a lot in the past. From Kottwitz \cite{Kottwitz1985, Kottwitz1997}, we know that the $\sigma$-conjugacy class of an element $b\in G(\breve F)$ is uniquely determined by two invariants, called its \emph{Kottwitz point} $\kappa(b)\in X_\ast(T)_{\Gamma}$ and its dominant Newton point $\nu(b) \in X_\ast(T)_{\Gamma_0}\otimes\mathbb Q$. Here, $T\subset G$ is a maximal torus, $X_\ast(T)$ is its cocharacter group, and $\Gamma_0$ resp.\ $\Gamma$ denote the absolute Galois groups of $\breve F$ resp.\ $F$.

There is a partial order on $B(G)$, whose properties have been studied by a very influential paper of Chai \cite{Chai2000}. A very peculiar connection between the set $B(G)$ and certain Coxeter elements in the Weyl group $W$ of $G$ has recently been introduced by Nie \cite{Nie2023}. This enabled Ivanov and Nie to study $p$-adic Deligne--Lusztig schemes of Coxeter type \cite{Ivanov2023, Ivanov2024}.

The second stratification is the \emph{Iwahori-Bruhat decomposition}
\begin{align*}
G(\breve F) = \bigsqcup_{x\in\tW} \mathcal I \dot{x}\mathcal I,
\end{align*}
where $\CI$ is a fixed $\s$-stable Iwahori subgroup and $\tW$ is the \emph{Iwahori subgroup}. Typically, the geometric of this decomposition are closely related to the combinatorial properties of the group $\tW$ and the representation theory of its Iwahori-Hecke algebra.

One may compare these two stratifications of $G(\breve F)$ by studying the intersections $[b]\cap \mathcal I\dot x\mathcal I$. The \emph{affine Deligne--Lusztig variety} associated to an element $x\in\tW$ and $b\in G(\breve F)$ is the reduced subscheme of the \emph{affine flag variety} $G(\breve F)/\mathcal I$ whose geometric points are given by
\begin{align*}
X_x(b) = \{g\in G(\breve F)/\mathcal I\mid g^{-1} b \sigma(g)\in \mathcal I\dot x\mathcal I\}.
\end{align*}
Note that $X_x(b)$ only depends on the $\sigma$-conjugacy class $[b]\in B(G)$. We may associate $X_x(b)$ to any pair $(x,[b]) \in \tW\times B(G)$.

In general, few things are known on the geometry of $X_x(b)$. While it is known that $X_x(b)$ must be finite-dimensional whenever it is non-empty, it is not known for which pairs $(x,[b])$ the affine Deligne--Lusztig variety is non-empty, and no general dimension formula is known that works for all these cases. The $\sigma$-stabilizer
\begin{align*}
\mathbb J_b(F) = \{g\in G(\breve F)\mid g^{-1} b\sigma(g) = b\}\subset G(\breve F)
\end{align*}
acts on $X_x(b)$ naturally by left multiplication, and it is known that $X_x(b)$ has only finitely many $\mathbb J_b(F)$-orbits of irreducible components. In hyperspecial level, there is a nice description of the $\mathbb J_b(F)$-orbits of irreducible components, as conjectured by M.\ Chen and X.\ Zhu \cite{Zhou2020, Nie2018, Takaya2022}. For Iwahori level however, the problem has proved to be very difficult, and no general answer is known.

Some partial answers are known to all of the aforementioned questions, and they have been fully answered in a few special cases in the past. For any given $x\in\tW$, the set
\begin{align*}
B(G)_x := \{[b]\in B(G)\mid X_x(b)\neq\emptyset\}
\end{align*}
has a unique minimal and a unique maximal element, and both have been explicitly computed \cite{Viehmann2021, Viehmann2014, Schremmer2022_newton}.
In \cite{Schremmer2022_newton}, the first named author has introduced the notion of \emph{length positivity} to describe the Newton point of the maximal element. 
This associates to every $x\in\widetilde W$ a natural subset $\LP(x)\subseteq W$, called the set of length positive elements of $x$, cf. \S\ref{sec:Palcove}.

Note that the Frobenius $\sigma$ acts on the set of simple reflections $S$ of the finite Weyl group $W$. By definition, a \emph{partial $\sigma$-Coxeter element} $w\in W$ is an element that can be written as a product of simple reflections $w = s_1\cdots s_n$ such that the $s_i$ lie in pairwise distinct $\sigma$-orbits.

Recall that $\tW$ decomposes as a semi-direct product $\tW = W\ltimes X_\ast(T)_{\Gamma_0}$, allowing us to write elements $x\in\widetilde W$ in the form $x = w\varepsilon^\mu$, where $w\in W$ and $\mu\in X_\ast(T)_{\Gamma_0}$.
Among all $v\in W$ such that $v^{-1}\mu$ is dominant, there is a uniquely determined one of minimal length in $W$. 
For this uniquely determined $v$, we define $\eta_\sigma(x) \coloneqq \sigma^{-1}(v)^{-1} wv$.
A recent article of He--Nie--Yu \cite{He2022} studies those $X_x(b)$ for which $\eta_\sigma(x)\in W$ is a partial $\sigma$-Coxeter element. In this case, one says that $x$ is an \emph{element with finite Coxeter part}.

In fact, the element $v$ constructed above always lies in $\LP(x)$.
From this perspective, it is natural to generalize the definition of an element with finite Coxeter part as follows (cf. Remark \ref{rmk:HNY}):
We say that an element $x = wt^{\mu}\in\tW$ is of positive Coxeter type, if $\sigma^{-1}(v)^{-1} wv$ is a partial $\s$-Coxeter element for some $v\in \LP(x)$. In this case, we say $(x,v)$ is a \emph{positive Coxeter pair} (cf. \S\ref{sec:Palcove}).
Theorem \ref{thm:introGeometry} below shows that the simple description of $X_x(b)$ obtained by He--Nie--Yu remains true after this generalization.

\subsection{Main Result}
We may assume that $G$ is quasi-split.
The first main result of this paper is analogous to the main result of He--Nie--Yu, but allowing to replace the element $\eta_\sigma(x)$ by $\sigma^{-1}(v)^{-1} wv$ for \emph{any} $v\in \LP(x)$. 

\begin{alphatheorem}[{Cf.\ Theorem~\ref{thm:fctConsequences}}]\label{thm:introGeometry}
Let $x=w\varepsilon^\mu \in\widetilde W$ and $v\in \LP(x)$ such that $(x,v)$ is a positive Coxeter pair, that is, $\sigma^{-1}(v)^{-1} w v$ is a partial $\sigma$-Coxeter element. Let $J = \supp_\sigma(\sigma^{-1}(v)^{-1} wv)$ denote its $\sigma$-support.
\begin{enumerate}[(a)]
%\item The unique minimal element $[b_{x,\min}]$ of $B(G)_x$ has the same Kottwitz point as $x$, and has Newton point equal to $\pi_J(v^{-1}\mu)$; here $\pi_J$ is the projection function from \cite[Definition~3.2]{Chai2000} or equivalently \cite[Section~3.1]{Schremmer2022_newton}.
%\item The unique maximal element $[b_{x,\max}]$ of $B(G)_x$ has the same Kottwitz point as $x$, and has Newton point equal to $\conv(v^{-1}\mu - \wt(v\Rightarrow \sigma(wv)))$. Here, $\conv$ is the convex hull function from \cite[Section~3.1]{Schremmer2022_newton} and $\wt(\cdot\Rightarrow\cdot)$ is the weight function of the quantum Bruhat graph. The $\lambda$-invariant of $[b_{x,\max}]$ in the sense of \cite{Hamacher2018} is given by $v^{-1}\mu - \wt(v\Rightarrow\sigma(wv))$.
\item The set $B(G)_x$ is given by the interval
\begin{align*}
B(G)_x = \{[b]\in B(G)\mid [b_{x,\min}]\leq [b]\leq [b_{x,\max}]\},
\end{align*}
where $[b_{x,\min}]$ and $[b_{x,\max}]$ are the minimal and the maximal element in $B(G)_x$, which are described in \S\ref{sec:3}. 
\item For any $[b]\in B(G)_x$, the affine Deligne--Lusztig variety $X_x(b)$ is equi-dimensional of dimension
\begin{align*}
\dim X_x(b) = \frac 12\Bigr(\ell(x) + \ell(\sigma^{-1}(v)^{-1} wv) - \langle \nu(b),2\rho\rangle - \defect(b)\Bigr).
\end{align*}
Here, $\ell(x)\in\mathbb Z_{\geq 0}$ is the Coxeter-theoretic length of $x$ and $\defect(b) = \rk_F(G) - \rk_F(\mathbb J_b)$.
\item For any $[b]\in B(G)_x$, $\mathbb J_b(F)$ acts transitively on the set of irreducible components of $X_x(b)$. Moreover, each irreducible component of $X_x(b)$ is an iterated fibre bundle, with $\mathbb G_m$ or $\mathbb A^1$ fibre, over an irreducible component of a classical Deligne--Lusztig variety of Coxeter type. Both the classical Deligne--Lusztig variety and the multiplicity of each fibre are explicitly described, analogously to \cite[Theorem~7.1]{He2022}.
\end{enumerate}
\end{alphatheorem}
\begin{comment}
\textbf{I am a little confused about the final statement. Is it true that the irreducible components of \emph{any} classical DLV $X_w$ are simply the Schubert cells $BwB/B$, i.e.\ isomorphic to the affine space $\mathbb A^{\ell(w)}$? -- Remark F.S.}

\textbf{I think it is not the case. For $\mathrm{SL}_2$, $X_1 = \BP^1(\BF_q)$ and $X_s = \BP^1(\overline{\BF}_q ) - \BP^1( \BF_q )$ -- Remark Q.Y.}
\end{comment}

The notion of elements of positive Coxeter type is much broader than He--Nie--Yu's notion of elements with finite Coxeter part. Firstly, elements of positive Coxeter type are stable under Deligne–Lusztig reduction, whereas elements of finite Coxeter type need not be. Second, we prove in Theorem \ref{thm:coxeterConjugacyClasses} that an element of minimal length in its $\s$-conjugacy class whose finite part is $\s$-conjugate to a partial $\s$-Coxeter element is of positive Coxeter type. But such an element, in most cases, does not have finite Coxeter part. This provides an important family of positive Coxeter type elements.

We remark that one may expect the iterated fibre bundle in part (c) of our above Theorem~\ref{thm:introGeometry} to be trivial. I.e.\ each irreducible component of $X_x(b)$ might be a \emph{direct product} of an irreducible component of a classical Deligne--Lusztig variety of Coxeter type with an appropriate number of $\mathbb A^1$ and $\mathbb G_m$. In particular, if $[b]$ is basic, then only $\mathbb A^1$--factors occur. We prove this latter result in \S\ref{sec:basic}. The more general statement has since been proved by Nie-Schremmer-Yu \cite{Nie2025}.

Our proof of Theorem \ref{thm:introGeometry} is more conceptual than the original proof by He--Nie--Yu, bringing several technical innovations and adding further insights. The key of our proof is that the set of length positive elements $\LP(x)$ and the positive Coxeter type property behave very well under Deligne-Lusztig reduction (cf. Lemma \ref{lem:reduction1} and Lemma \ref{lem:reduction2}). This allows us to use induction on $\ell(x)$, which makes the proof completely straightforward. 

\subsection{Applications}

One application of Theorem~\ref{thm:introGeometry} is the study of basic loci of Shimura varieties.
The intersection of an EKOR (Ekedahl--Kottwitz--Oort--Rapoport) stratum in a Shimura variety with its basic locus can be described in terms of $X_x(b)$ for a basic $b$, up to a finite morphism (cf. \cite{Goertz2019,}).
Basic loci of Coxeter type have been introduced and studied by G\"{o}rtz-He-Nie, cf.\ \cite{GH15, Goertz2019, GHN22}.
They described basic loci of Shimura varieties by exploiting the case where $X_x(b)$ decomposes as a union of Deligne–Lusztig varieties of Coxeter type.
The notion of positive Coxeter type can be used to generalize the cases of Coxeter type, cf.\ \cite{Shimada24, ST24} for the cases of $\mathrm{GL}_n$ and $\mathrm{GSp}_6$ respectively.
In many cases treated in \cite{Shimada24, ST24}, there exists an EO stratum of positive Coxeter type, but not of finite Coxeter type.
This suggests necessity of a generalization of the notion of finite Coxeter type.
For example, the second named author proves that the EO strata associated with $(\mathrm{GL}_n, (1,1,0,\ldots,0))$ and basic $b$ are of positive Coxeter type. In this case, all other strata except one are not of finite Coxeter type if $n\geq 4$.
Moreover, for the basic locus of the $\mathrm{GU}(2,n-2)$ Shimura variety at a prime inert, there exists an irreducible component that can be written as the closure of an EO stratum of positive Coxeter type (cf. \cite[Remark 3.6]{Shimada24b}); this refines the description given in \cite{FHI26}.

The notion of positive Coxeter type also plays a crucial role in \cite{SV25}. The first author together with Viehmann introduces the notion of \emph{depth} of a Shimura datum, with the Shimura varieties of depth $\leq 1$ being the famous class of fully Hodge-Newton decomposable ones \cite{Goertz2019}. In their study of Ekedahl-Oort strata in case of depth $<2$, many cases can only be handled by proving the positive Coxeter type property.

Another potential application of Theorem~\ref{thm:introGeometry} is the study of affine Lusztig varieties. These encode certain orbital integrals of conjugacy classes in the loop group, and are studied to obtain a theory of affine character sheaves \cite{He2023_affinelusztig}. Following He's correspondence between affine Lusztig varieties and affine Deligne--Lusztig varieties, we expect that affine Weyl group elements of positive Coxeter type may also have the most beautiful affine Lusztig varieties attached to them.

\subsection{A Converse to Theorem~\ref{thm:introGeometry}}
Our second main result provides a converse to Theorem~\ref{thm:introGeometry}. If one knows that the conclusion of Theorem~\ref{thm:introGeometry} holds for the minimal $\sigma$-conjugacy class $[b] = [\dot x]$, then one can conclude that $x$ of positive Coxeter type. Explicitly:

\begin{alphaproposition}[{Cf.\ Proposition~\ref{prop:pctCharacterization}}]\label{thm:introConverse}
Let $x = w\varepsilon^\mu$ and denote its Newton point by $\nu(\dot x)\in X_\ast(T)_{\Gamma_0}\otimes\mathbb Q$. Consider the closed subset $X_{\underline p}$ of $X_x(\dot x)$ obtained from the Deligne--Lusztig reduction by iteratively choosing closed subsets along a length decreasing sequence of cyclic shifts. Then $x$ has positive Coxeter type if and only if the following two conditions are both satisfied.
\begin{enumerate}[(i)]
\item The element $w\in W$ is $\sigma$-conjugate to a partial $\sigma$-Coxeter element in $W$.
\item There exists some $v\in \LP(x)$ with
\begin{align*}
\dim X_{\underline p} = \frac 12\left(\ell(x)+\ell(v^{-1}\sigma(wv)) - \langle \nu(\dot x),2\rho\rangle - \defect([\dot x])\right).
\end{align*}
\end{enumerate}
If both conditions are satisfied and $v$ is as in (ii), then $(x,v)$ is a positive Coxeter pair.
\end{alphaproposition}
The paper is organized as follows. In \S\ref{sec:preliminary}, we introduce basic notions such as affine Deligne–Lusztig varieties and Deligne–Lusztig reduction, which is crucial for reducing many arguments to smaller cases via induction on the length of $x$. In \S\ref{sec:3}, we give a complete description of the minimal and the maximal elements of $B(G)_x$ (cf. Proposition \ref{prop:minNewton} and Proposition \ref{prop:finiteCoxGNP}), using length positive elements. We also use the weight function of the quantum Bruhat graph to explicitly describe the maximal element of $B(G)_x$. In the context of our theorem, this function is very straightforward to compute, cf.\ \cite[Proof of Theorem~1.2~(b)]{Milicevic2020} and Remark~\ref{rem:qbgForCoxeter} below. This explicit description of the generic Newton point, even for elements with finite Coxeter parts, is not contained in the original article of He--Nie--Yu. In \S\ref{sec:reduction} and \S\ref{sec:main}, we prove Theorem \ref{thm:introGeometry} using the description in \S\ref{sec:3}.
The remainder of the paper is devoted to more detailed studies of $X_x(b)$.
In \S\ref{sec:veryspecial}, we give a more precise account of the geometry of $X_x(b)$ using the notion of very special parahoric subgroups (cf. Corollary \ref{cor:geo}).
In \S\ref{sec:minimallengthcase}, we prove a characterization of elements of positive Coxeter type in the minimal-length case, and in \S\ref{sec:converse} we prove Proposition~\ref{thm:introConverse}.
We show in \S\ref{sec:basic} that the iterated fiber bundle in Theorem~\ref{thm:introGeometry} is trivial in the basic case.
%Finally, in \S\ref{sec:endpoint2}, we prove Theorem \ref{thm:introClassPolynomials}, giving a more detailed description of the minimal length cases.

%We remark that if $x$ is of positive Coxeter type, then $[\dot x]$ is actually the unique minimal element in $B(G)_x$, and that it is common and easy to reduce to the case where this minimal element is basic.
\subsection*{Acknowledgements}
FS is partially supported by the New Cornerstone Foundation through the New Cornerstone Investigator grant, and by Hong Kong RGC grant 14300122, both awarded to Prof.\ Xuhua He.
RS was supported by the WINGS-FMSP program at the Graduate School of Mathematical Science, the University of Tokyo. RS was also supported by JSPS KAKENHI Grant number JP21J22427. QY is partially supported by the National Natural Science Foundation of China (grant no. 12501018). We thank Michael Rapoport for helpful comments on a preliminary version of this article.
\section{Preliminary}\label{sec:preliminary}
\subsection{Affine Weyl group}
We fix a non-archimedian local field $F$ whose completion of the maximal unramified extension will be denoted $\breve F$. We write $\mathcal O_F$ and $\mathcal O_{\breve F}$ for the respective rings of integers. Let $\mathbf k$ be the residue field of $\breve F$. Pick a uniformizer $\varepsilon \in F$. The Galois group $\Gal(\breve F/F)$ is generated by the Frobenius $\sigma$.

We consider a quasi-split connected reductive group $G$ over $F$. The assumption of $G$ being quasi-split is only made to lighten the notational burden, and everything developed can be immediately generalized to arbitrary connected and reductive groups via \cite[Section~2]{Goertz2015}.

We construct its associated affine root system and affine Weyl group following Haines-Rapoport \cite{Haines2008} and Tits \cite{Tits1979}.

Fix a maximal $\breve F$-split torus $S\subseteq G_{\breve F}$ and  write $T$ for its centralizer in $G_{\breve F}$, so $T$ is a maximal torus of $G_{\breve F}$. Write $\mathcal A = \mathcal A(G_{\breve F},S)$ for the apartment of the Bruhat-Tits building of $G_{\breve F}$ associated with $S$. We pick a $\sigma$-invariant alcove $\mathfrak a$ in $\mathcal A$. This yields a $\sigma$-stable Iwahori subgroup $\CI\subset G(\breve F)$.

Denote the normalizer of $T$ in $G$ by $N(T)$. Then the quotient
\begin{align*}
\widetilde W = N_G(T)(\breve F) / (T(\breve F)\cap \CI)
\end{align*}
is called \emph{extended affine Weyl group}, and $W = N_G(T)(\breve F)/T(\breve F)$ is the \emph{(finite) Weyl group}. The Weyl group $W$ is naturally a quotient of $\widetilde W$. We write $\cl : \widetilde W\to W$ for the natural projection map.

The affine roots as constructed in \cite[Section~1.6]{Tits1979} are denoted $\Phi_\af$. Each of these roots $a\in \Phi_\af$ defines an affine function $a:\mathcal A\rightarrow\mathbb R$. The vector part of this function is denoted $\cl(a) \in V^\ast$, where $V = X_\ast(T)_{\Gamma_0}\otimes \mathbb R$. Here, $\Gamma_0 = \Gal(\overline F/\breve F)$ is the absolute Galois group of $\breve F$, i.e.\ the inertia group of $\Gamma = \Gal(\overline F/F)$. The set of \emph{(finite) roots} is the root system\footnote{This is different from the root system that \cite{Tits1979} and \cite{Haines2008} denote by $\Phi$; it coincides with the reduced root system called $\Sigma$ in \cite{Haines2008}.  } $\Phi := \cl(\Phi_\af)$.

The affine roots in $\Phi_\af$ whose associated hyperplane is adjacent to our fixed alcove $\mathfrak a$ and which send the points of $\mathfrak a$ to non-negative real numbers are called \emph{simple affine roots} and denoted $\Delta_\af\subseteq \Phi_\af$.% We denote $\Delta$ the set \emph{finite simple roots}.

Writing $W_\af$ for the extended affine Weyl group of the simply connected quotient of $G$, we get a natural $\sigma$-equivariant short exact sequence (cf.\ \cite[Lemma~14]{Haines2008})
\begin{align*}
1\rightarrow W_\af\rightarrow\widetilde W\rightarrow \pi_1(G)_{\Gamma_0}\rightarrow 1.
\end{align*}
Here, $\pi_1(G) := X_\ast(T)/\mathbb Z\Phi^\vee$ denotes the Borovoi fundamental group.

For each $x\in \widetilde W$, we denote by $\ell(x)\in \mathbb Z_{\geq 0}$ the length of a shortest alcove gallery from $\mathfrak a$ to $x\mathfrak a$. The elements of length zero are denoted $\Omega$. The above short exact sequence yields an isomorphism of $\Omega$ with $\pi_1(G)_{\Gamma_0}$, realizing $\widetilde W$ as semidirect product $\widetilde W = \Omega\ltimes W_\af$.

Each affine root $a\in \Phi_\af$ defines an affine reflection $r_a$ on $\mathcal A$. The group generated by these reflections is naturally isomorphic to $W_\af$ (cf.\ \cite{Haines2008}), so by abuse of notation, we also write $r_a\in W_\af$ for the corresponding element. We define $\BS_\af := \{r_a\mid a\in \Delta_\af\}$, called the set of \emph{simple affine reflections}. The pair $(W_\af, \BS_\af)$ is a Coxeter group with length function $\ell$ as defined above. %Define $\BS :=\{s_\a\mid \a\in \Delta\}$, called the set of \emph{finite simple reflections}

We pick a special vertex $\mathfrak x\in \mathcal A$ that is adjacent to $\mathfrak a$. Since we assumed $G$ to be quasi-split, we may and do choose $\mathfrak x$ to be $\sigma$-invariant. We identify $\mathcal A$ with $V$ via $\mathfrak x\mapsto 0$. This allows us to decompose $\Phi_\af = \Phi\times\mathbb Z$, where $a = (\alpha,k)$ corresponds to the function
\begin{align*}
V\rightarrow \mathbb R, v\mapsto \alpha(v)+k.
\end{align*}
From \cite[Proposition~13]{Haines2008}, we moreover get decompositions $\widetilde W = W\ltimes X_\ast(T)_{\Gamma_0}$ and $W_\af = W\ltimes \mathbb Z\Phi^\vee$. Using this decomposition, we write elements $x\in \widetilde W$ as $x = w\varepsilon^\mu$ with $w\in W$ and $\mu\in X_\ast(T)_{\Gamma_0}$. For $a = (\alpha,k)\in \Phi_\af$, we have $r_a = s_\alpha \varepsilon^{k\alpha^\vee}\in W_\af$, where $s_\alpha\in W$ is the reflection associated with $\alpha$. The natural action of $\widetilde W$ on $\Phi_\af$ can be expressed as
\begin{align*}
(w\varepsilon^\mu)(\alpha,k) = (w\alpha,k-\langle\mu,\alpha\rangle).
\end{align*}
We define the dominant chamber of $V$ to be the Weyl chamber containing our fixed alcove $\mathfrak a$. %\remind{Do we use this notion $C$ in the paper? -- QY}
This gives a Borel subgroup $B\subseteq G$, and corresponding sets of positive/negative/simple roots $\Phi^+, \Phi^-, \Delta\subseteq \Phi$. The set of \emph{(finite) simple reflections} is $\BS :=\{s_\a\mid \a\in \Delta\}\subset W$.

By abuse of notation, we denote by $\Phi^+$ also the indicator function of the set of positive roots, i.e.
\begin{align*}
\Phi^+:\Phi\rightarrow\{0,1\},\quad \alpha\mapsto\begin{cases}1,&\alpha\in \Phi^+,\\
0,&\alpha\in \Phi^-.\end{cases}
\end{align*}
The sets of positive and negative affine roots can be defined as
\begin{align*}
\Phi_\af^+:=&(\Phi^+\times \mathbb Z_{\geq 0})\sqcup (\Phi^-\times \mathbb Z_{\geq 1}) = \{(\alpha,k)\in \Phi_\af\mid k\geq \Phi^+(-\alpha)\},
\\\Phi_\af^- :=&-\Phi_\af^+ = \Phi_\af\setminus \Phi_\af^+= \{(\alpha,k)\in \Phi_\af\mid k< \Phi^+(-\alpha)\}.
\end{align*}
One checks that $\Phi_\af^+$ consists of precisely those affine roots which are sums of simple affine roots.

Decompose $\Phi$ as a direct sum of irreducible root systems, $\Phi = \Phi_1\sqcup\cdots\sqcup \Phi_r$. Each irreducible factor contains a uniquely determined longest root $\theta_i\in \Phi_i^+$. Now the set of simple affine roots is
\begin{align*}
\Delta_\af = \{(\alpha,0)\mid \alpha\in \Delta\}\sqcup\{(-\theta_i,1)\mid i=1,\dotsc,r\}\subset \Phi_\af^+.
\end{align*}

We call an element $\mu\in X_\ast(T)_{\Gamma_0}\otimes \mathbb Q$ \emph{dominant} if $\langle \mu,\alpha\rangle\geq 0$ for all $\alpha\in \Phi^+$. 
\begin{comment}
Similarly, we call it $C$-regular for a real number $C$ if
\begin{align*}
\abs{\langle\mu,\alpha\rangle}\geq C
\end{align*}
for each $\alpha\in \Phi^+$.

An element $x = w\varepsilon^\mu\in \widetilde W$ is called $C$-regular if $\mu$ is. 
    
\end{comment}

For elements $\mu,\mu'$ in $X_\ast(T)_{\Gamma_0}\otimes \mathbb Q$ (resp.\ $X_\ast(T)_{\Gamma_0}$ or $X_\ast(T)_{\Gamma}$), we write $\mu\leq \mu'$ if the difference $\mu'-\mu$ is a $\mathbb Q_{\geq 0}$-linear (resp.\ $\mathbb Z_{\geq 0}$-linear) combination of positive coroots.

The support $\supp(w)\subseteq \Delta$ of an element $w\in W$ is defined to be the set of those simple roots $\alpha$ such that the corresponding simple reflection $s_\alpha$ occurs in some (or equivalently any) reduced expression of $w$. The $\s$-support of $w$ is defined as
\begin{align*}
\supp_\s(w) = \bigcup_{i\in \BZ} \s^i(\supp(w)).
\end{align*}
\begin{comment}
We write $\LP(x)\subseteq W$ for the set of length positive elements as introduced in \cite[Section~2.2]{Schremmer2022_newton}. If $x$ is $2$-regular, then $\LP(x)$ consists only of one element, namely the uniquely determined $v\in W$ such that $v^{-1}\mu$ is dominant.
    
\end{comment}

\subsection{$\s$-conjugacy classes}\label{sec:1.2}
Two elements $g,g'\in G(\breve F)$ are called \emph{$\s$-conjugate}, if $g' = h^{-1}g\s(h)$ for some $h\in G(\breve F)$. For $b \in G(\breve F)$, we denote $[b]=\{h^{-1}b\s(h)\mid  h\in G(\breve F)\}$ and $B(G)=\{[b]\mid b\in G(\breve F)\}$.

We recall the well known Kottwitz's map
 $$\k : B(G)\rightarrow \pi_1(G)_{\G},$$
 and the Newton map
 $$\nu : B(G)\rightarrow X_*(T)_{\G_0} \otimes \BQ.$$
 
Kottwitz proved that a $\s$-conjugacy class $[b]$ is completely determined by $(\k(b),\nu(b))$ \cite{Kottwitz1985, Kottwitz1997}. Let $\le$ denote the partial order on $B(G)$. Note that $[b]\le [b']$ if and only if $\k(b) = \k(b') $ and $\nu({b}) \le\nu({b'})$.

For any $\nu\in X_*(T)_{\G_0}\otimes\BR$, denote $I(\nu) = \{ \a\in\Delta\mid \<\nu,\a\>=0   \}$.

%We recall the parametrization of $\sigma$-conjugacy classes in $\widetilde W$ as developed by He--Nie (cf.\ \cite{He2014b, He2014, He2015, He2015b} and many other fundamental publications by these authors).

Consider  a $\s$-stable subset  $J  $ of $\BS$. Let $\tW_J=W_J\ltimes X_*(T)_{\G_0} $ be the Iwahori--Weyl group of the standard Levi subgroup $M_J$ of $G$. We write ${}^{J}\tW$ for the set of minimal length representatives for cosets in $W_J\backslash \tW$. By $  J_\af \supseteq J$ we denote the set of simple affine roots for $\tW_J$.

\subsection{Affine Deligne-Lusztig varieties}\label{sec:1.3}
Let $x\in \tW$ and $[b]\in B( G)$. The affine Deligne-Lusztig variety $X_x(b)$ is defined as
\begin{align*}
X_x(b) = \{g\CI  \in G(\breve F)/\CI\mid g^{-1}b \sigma(g) \in \CI x \CI\}.
\end{align*}
Set $B( G)_x = \{[b]\in B( G)\mid X_x(b)\ne \emptyset\}$. 
Note that $X_x(b)$ is a locally closed subscheme of the affine flag variety $G(\breve F)/\CI$ (which itself is only an ind-scheme). Note also that $X_x(b)$ is finite dimensional $\mathbf k$-scheme, locally of finite type. Let $\BJ_b(F) = \{ g\in G(\breve F)\mid g^{-1}b\s(g) = b\}$, the $\s$-centralizer of $b$. It is easy to see that $\BJ_b(F)$ acts on $X_x(b)$ via left multiplication.

Let $x\in \tW$ and $s\in \BS_\af$. If $\ell(sx\s(s))\le \ell(x)$, we write $x\xrightarrow{s }_\s sx\s(s)$. We write $x\rightarrow_\s x'$ for some $x'\in \tW$ if there exist $x_0=x,x_1,x_2,\ldots,x_r=x'$ and $s_1,s_2,\ldots,s_r$ such that $x_0\xrightarrow{s_1 }_\s x_1 \xrightarrow{s_2 }_\s x_2\xrightarrow{s }_\s\cdots \xrightarrow{s_r }_\s x_r$. We write $x\approx_\s x'$ if $x\rightarrow_\s x'$ and $x'\rightarrow_\s x$. %Moreover, we write $x\tilde\approx_s x'$ if $x\approx_\s \t^{-1}x'\s(\t)$ for some $\t\in\mathcal O$.
For any $\tilde \CO\in B(\tW)_\s$, denote by $\tilde \CO_{\min}$ the subset of elements of minimal length.

He and Nie prove in \cite[Theorem 2.10]{He2014b} that for any element $x\in \tW$, there exists some element $x'$ which is of minimal length in its $\s$-conjugacy class, such that $x\rightarrow_\s x'$. He proves in \cite[Theorem 3.5]{He2014} that if $x$ is of minimal length in its $\s$-conjugacy class, then $\CI x \CI $ is contained in a unique $\s$-conjugacy class $[b] \in B(G)$, in other words, the set $B(G)_x$ contains only one element. For such minimal length elements, He gives an explicit description for the corresponding affine Deligne-Lusztig varieties \cite[Theorem 4.8]{He2014}.
%\begin{theorem}\label{thm:minimal-length}
%Let $x\in\tW$ be a minimal length element. Let $[b] \in B(G)_x$. Then
%$$X_x(b)\simeq \BJ_b(F) \times_{\CP}X'.$$
%Here, $\CP$ is a parahoric subgroup of $\BJ_b(F)$ and $X'$ is a classical Deligne-Lusztig variety in the reductive quotient of $\CP$.
%\end{theorem}

\begin{comment}
Let $x$ be a minimal length element. By \cite[Theorem~3.4]{He2014b}, there exist $J\subset \BS_\af$ with $W_J$ finite, a fundamental element $x_{str}$ with $x_{str}\in {}^{J}\widetilde{W}^{\sigma (J)}$ and $u\in W_J$ such that $w \tilde{\sim}_{\sigma}ux_{str}$. Let $\mathcal P$ be the standard parahoric subgroup corresponding to $J$ and $\bar{\mathcal P}$ be its reductive quotient. Let $\bar{ I}$ be the reductive quotient of $\bar{\mathcal P}$. Let $\sigma' = \text{Ad}(x)\circ \sigma$. Let $[b] = [x]$. By [He14, Theorem 4.8] \cite[Theorem~4.8]{He2014} , we have
$$X_x(b) \simeq \mathbb{J}_b \times_{\mathbb{J}_b\cap \mathcal P} X^{\bar{\mathcal P}}_{u} ,$$
where $$X^{\bar{\mathcal P}}_{u} = \{ \bar{p}\in \bar{\mathcal P}/\bar{  I}\mid \bar{p}^{-1} \sigma'(\bar{p}) \in \bar{ I}u\bar{ I}\} ,$$ the classical Deligne-Lusztig variety corresponding to $u$ inside $\mathcal P$. Suppose $u$ is elliptic in $W_J$, then $\mathbb J_b$ act transtively on the set of irreducible components of $X_x(b)$.
    
\end{comment}
\subsection{Deligne-Lusztig reduction }\label{sec:1.4}\label{sec:dlreduction}

In this subsection, we recall the Deligne-Lusztig reduction method established in \cite[Corollary 2.5.3]{Goertz2010b} and the notion of reduction tree of elements of $\widetilde W$ introduced in \cite[\S3.4]{He2022}.

\begin{proposition}\label{DL-reduction}
Let $x\in \widetilde W$ and let $a\in \Delta_\af$.
If $\mathrm{ch}(F)>0$, then the following two statements hold for any $b\in G(\breve F)$.
\begin{enumerate}[(1)]
\item If $\ell(r_a x r_{\sigma a})=\ell(x)$, then there exists a $\mathbb J_b(F)$-equivariant universal homeomorphism $X_x(b)\rightarrow X_{r_a x r_{\sigma a}}(b)$.
\item If $\ell(r_a x r_{\sigma a})=\ell(x)-2$, then there exists a decomposition $X_x(b)=X_1\sqcup X_2$ such that
\begin{itemize}
\item $X_1$ is open and there exists a $\mathbb J_b(F)$-equivariant morphism $X_1\rightarrow X_{r_a x}(b)$, which is  the composition of a Zariski-locally trivial $\mathbb G_m$-bundle and a universal homeomorphism. 
\item $X_2$ is closed and there exists a $\mathbb J_b(F)$-equivariant morphism $X_2\rightarrow X_{r_a x r_{\sigma a}}(b)$, which is the composition of a Zariski-locally trivial $\mathbb A^1$-bundle and a universal homeomorphism. 
\end{itemize}
\end{enumerate}
If $\mathrm{ch}(F)=0$, then the above statements still hold by replacing $\mathbb A^1$ and $\mathbb G_m$ by $\mathbb A^{1,\mathrm{pfn}}$ and $\mathbb G_m^{\mathrm{pfn}}$ respectively.
\end{proposition}

In case (2) of the Proposition, we call the reduction $x\rightharpoonup r_ax$ type I reduction and call $x\rightharpoonup r_ax r_{\sigma(a)}$ type II reduction.

A reduction tree of $x$ is a tree whose vertices are elements of $\widetilde W$ and edges are of the form $y\rightharpoonup z$, such that there are some $y'$ with $y' \approx_{\sigma} y$ and a type I or type II reduction $y'\rightharpoonup z$. We call such an edge type I edge or type II edge if $y'\rightharpoonup z$ is type I or type II reduction.

A reduction tree of $x$ is constructed inductively as follows. If $x$ is of minimal length in its $\s$-conjugacy class, then the tree consists of a single vertex $x$. if $x$ is not of minimal length, we can find some $x'$ with $x'\approx_{\sigma}x$ and some simple affine reflection $r_a$ with $\ell(r_ax'r_{\sigma a})<\ell(x')$. Then the reduction tree of $x$ consists of the reduction tree of $r_ax'$, the reduction tree of $r_ax'r_{\sigma a}$, the edge $x\rightharpoonup r_ax'$ and the edge $x\rightharpoonup r_ax'r_{\sigma a}$. 

Note that reduction trees of $x$ are far from unique. In fact, the number of reduction trees of an element can be very large.

Let $\mathcal{T}$ be a reduction tree of $x$. An \emph{endpoint} or a \emph{leaf} of $\mathcal{T}$ is a vertex $y$ such that there is no edge starting from $y$. For a (complete) path $\underline p: x=x_0\rightharpoonup x_1\rightharpoonup\cdots\rightharpoonup x_n$, where $x_n$ is an endpoint of $\mathcal{T}$, denote $  \text{end}(\underline p)=x_n$. Let $\ell_I(\underline p)$ and $\ell_{II}(\underline p)$ be the numbers of type I edges and type II edges of $\underline p$. Denote by $[b]_{\underline p}$ the $\sigma$-conjugacy class of $G(\breve F)$ containing to $\mathcal I\text{end}(\underline p)\mathcal I$. 

We have, for any $[b]$ and $\mathcal{O}$,

$$X_x(b) = \bigsqcup_{\underline p:\text{path of } \mathcal{T}; [b]_{\underline p } = [b]}X_{\underline p }$$
where $X_{\underline p}$ is the variety given by Proposition \ref{DL-reduction}.

For any $a_1, a_2 \in \BN$, we say that a scheme $X$ is an \textit{iterated fibration} of type $(a_1, a_2)$ over a scheme $Y$ if there exist morphisms $$X=Y_0 \to Y_1 \to \cdots \to Y_{a_1+a_2}=Y$$ such that for any $i$ with $0 \le i<a_1+a_2$, $Y_i$ is a Zariski-locally trivial $\BA^{1, (\mathrm{perf})}$-bundle or $\BG^{(\mathrm{perf})}_m$-bundle over $Y_{i+1}$, and there are exactly $a_1$ locally trivial $\BG^{(\mathrm{perf})}_m$-bundles in the sequence. In this case, there are exactly $a_2$ locally trivial $\BA^{1, (\mathrm{perf})}$-bundles in the sequence. Then the variety $ X_{\underline p }$ is $\BJ_b(F)$-equivalently universally homeomorphic to an iterated fibration of type $(\ell_I(\underline p),\ell_{II}(\underline p) )$ over $X_{\mathrm{end}(\underline p)}(b)$. Moreover
$$\dim X_{\underline p} = \ell_I(\underline p) + \ell_{II}(\underline p) + \ell(\mathrm{end}(\underline p))-\<\nu(b),2\rho\>.$$
Here, $2\rho\in \mathbb Z\Phi$ is the sum of all positive roots.

\section{$P$-alcove elements and length positivity}\label{sec:Palcove}
\begin{definition}[{\cite[Definition 2.5]{Schremmer2022_newton}}]
Let $x=w\varepsilon^{\mu}\in\tW$ and $\a\in\Phi^+$. We define the length functional of $x$ as
$$\ell(x,\a) = -\Phi^+(w\a) + \<\mu,\a\> + \Phi^+(\a)\in\mathbb Z.$$
\end{definition}
The length functional $\ell(x,\a)$ counts, in a signed way, the number of positive affine roots $(\a,k)\in \Phi_\af^+$ whose image $x(\a,k)\in\Phi_\af$ is a negative affine root.

\begin{definition}[{\cite[\S3.3]{Goertz2015}}]\label{def:Palcove}
Let $J$ be a $\s$-stable subset of $\BS$ and $u\in W$. We say an element $x\in \tW$ is a $(J,u,\s)$-alcove element, if the following conitions are both satisfied:
\begin{enumerate}[(a)]
    \item The element $u^{-1}x\s(u)$ lies in $\tW_J$, and
    \item we have $\ell(x,w^{-1}u\b) \ge 0 $ for all\footnote{The translation of the original formulation to the length functional is straightforward but not immediate, partly due to the different choices of Borel subgroups in \cite{Goertz2015} and our paper.} $\b\in\Phi^+ \setminus \Phi_J$.
\end{enumerate}
We say that $x$ is a \emph{normalized} $(J,u,\s)$-alcove element if $u\in W^J$ \cite[\S4]{Viehmann2021}.
\end{definition}

Note that $x$ is a $(J,u,\s)$-alcove element if and only if it is a $(J,uu',\s)$-alcove element for any $u'\in W_J$. 

\begin{definition}[{\cite[Definition 2.7]{Schremmer2022_newton}}]
Let $v\in W$ and $x\in \tW$. We say that $v$ is a length positive element of $x$, if $\ell(x,v\a)\ge0$ for all $\a\in\Phi^+$. Denote by $\LP(x)$ the set of all length positive element for $x$.

\end{definition}
We can see from the definition that the length positivity condition is a refinement of the second condition in the definition of $(J,u,\s)$-alcove element. In view of \cite[Lemma~2.3]{Schremmer2022_newton}, we conclude that $x$ is a $(J,u,\delta)$-alcove element if and only if there exists some $v\in \LP(x)$ with $v\in w^{-1} u W_J$ and $v^{-1}\sigma(wv)\in W_J$.

For any $x=w\varepsilon^{\mu}$, the set $\LP(x)$ is non-empty. Indeed, $\LP(x)$ contains the minimal element $v\in W$ such that $v^{-1}\mu$ is dominant. The element $v$ and the set $J$ in the definition are not unique. Hence we call $J$ the support of the pair $(x,v)$ instead of $x$. Nevertheless, the cardinality of $J$ is uniquely determined by $x$. See \S\ref{sec:support}.

In the case where $x=w\varepsilon^{\mu}$ is 2-regular, i.e., $\left\vert \<\mu,\a\> \right\vert \ge2$ for any $\a\in\Phi^+$, then by \cite[Proposition~2.15]{Schremmer2022_newton}, the set $\LP(x)$ consists only of one element.

\begin{definition}
 Let $x=w\varepsilon^{\mu}\in \tW$. We say $x$ is a positive Coxeter type element, or $x$ is of positive Coxeter type, if there is $v\in \LP(x)$ and a $\s$-stable subset $J\subset \BS$ such that $v^{-1}\s(wv)$ is a $\s$-Coxeter element in $W_J$. In this case, we say $(x,v)$ is a positive Coxeter pair with support $J$.
\end{definition}

%Observe that if $(x,v)$ is a positive Coxeter pair, then $x$ is a $(J,u,\delta)$-alcove for $u\in wvW_J$. We have the following partial converse: If $x$ is a normalized $(J,u,\delta)$-alcove element and $\tilde x = u^{-1} x\sigma(u)\in \widetilde W_J$, then these elements are closely related in terms of their Newton strata and class polynomials \cite{Viehmann2021, He2015b}. Moreover, we have the following implication:

\begin{comment}
\begin{lemma}\label{lem:PalcovePositiveCoxeter}
Let $x$ be a normalized $(J,u,\delta)$-alcove element and $\tilde x = u^{-1} x\sigma(u)\in\widetilde W$. Let $v\in W_J$ such that $(\tilde x,v)$ are a positive Coxeter pair for $M = M_J$ with support $J'\subseteq J$. Then $(x,\sigma(u)v)$ is a positive Coxeter pair for $\mathbf G$ with support $J'$.
\end{lemma}
\begin{proof}
The only non-trivial statement is $\sigma(u)v\in \LP(x)$. For this, let $\alpha\in \Phi^+$. If $\alpha\notin \Phi_J$, then we get
\begin{align*}
\ell(x,\sigma(u)v\alpha) = \ell(x,w^{-1} u\cdot u^{-1} w\sigma(u)v\alpha)\geq 0
\end{align*}
since $u^{-1} w\sigma(u)$ and $v$ lie in $W_J$. If $\alpha\in \Phi_J^+$, we calculate using \cite[Lemma~2.12]{Schremmer2022_newton} that $0\leq \ell(\tilde x,v\alpha) = \ell(x,\sigma(u) v\alpha)$.
\end{proof}
    
\end{comment}

%We will later see a partial converse to this lemma, namely $x$ being of positive Coxeter type implies the same for $\tilde x$.
It is known that if $x\in\widetilde W$ has length zero and corresponds to a superbasic $\sigma$-conjugacy class, then $x$ is of positive Coxeter type \cite[Lemma~3.10]{Schremmer2022_newton}. We shall prove that all minimal length element in a partial $\s$-Coxeter $\sigma$-conjugacy class are of positive Coxeter type. For example if $G = \GL_n$, then every element that has minimal length in its $\sigma$-conjugacy class is of positive Coxeter type. See \S~\ref{sec:minimallengthcase}.

%By Lemma~\ref{lem:PalcovePositiveCoxeter} and ???, it follows that every fundamental (also known as $\sigma$-straight) element in $\widetilde W$ is of positive Coxeter type. 
\begin{remark}\label{rmk:HNY}
Let $x=w\varepsilon^{\mu}$ and $v$ be the minimal element such that $v^{-1}(\mu)$ is dominant. Suppose $v^{-1}w\s(v)$ is a partial $\s$-Coxeter element. He, Nie and Yu \cite{He2022} refer to these kind of elements as \enquote{elements with finite Coxeter part}. Although this $v$ is also used in the definition of the virtual dimension in \S\ref{sec:converse}, it exhibits some undesired behaviour on certain regions of the Bruhat--Tits buildings known as the \emph{critical strips}. For an example of such pathologies, note that the properties of affine Deligne--Lusztig varieties are invariant under automorphisms of the affine Dynkin diagram of $\widetilde W$, or by replacing $(x,[b],\sigma)$ by $(x^{-1}, [b^{-1}], \sigma^{-1})$. However, both the virtual dimension and the notion of finite Coxeter part are not stable under these procedures, whereas the notion of positive Coxeter type introduced in this article is. For a concrete example, one may study the group $G = \GL_3$, and observe that two of the three simple affine reflections in $\widetilde W$ have finite Coxeter part, whereas the third does not. 

 In order to remedy these pathologies, the first named author has introduced the notion of \emph{length positivity} \cite{Schremmer2022_newton}. This associates to every $x\in\widetilde W$ a natural subset $\LP(x)\subseteq W$, called the set of length positive elements of $x$, such that the element $v\in W$ constructed above is always an element of $\LP(x)$.
\end{remark}

%We list some properties of length positivity here.

Recall that for any $x\in \tW$, we define $B(G)_x=\{[b]\in B(G)\mid X_x(b)\ne\emptyset\}$. 

The following important property for (normalized) $(J,u,\s)$-alcove elements is first proved in \cite[\S8]{Goertz2010} for split groups, see also \cite[Theorem 2]{Sch23} for the general statement and a different proof.
\begin{theorem}\label{thm:Palcove}
Let x be a normalized $(J,u,\s)$-alcove element. Let $\tilde x = u^{-1}x\s(u)\in \tW_J$. Write $M=M_J$. Then the restriction of the natural map $B(M )\rightarrow B(G) $ gives a bijection $B(M )_{\tilde{x}}\rightarrow B(G)_x $. Moreover, the $M $-dominant Newton point $\nu_{M }(b)$ equals the $G $-dominant Newton point $\nu(b)$ for any $[b]\in B(M)_{\tilde x}$.
\end{theorem}

For any element $\l\in V=X_*(T)_{\G_0}\otimes \BR$, let $\avg_\s(\l)$ be the $\s$-average of $\l$. For any $\s$-stable subset $J\subset \BS$, we define \cite[Definition~3.2]{Chai2000} $$\pi_J(\l) = \avg_\s \frac{1}{\sharp W_J} \sum_{w\in W_J}w(\l ) .$$

The following facts on length positivity will be needed later in this article.
\begin{lemma}\label{lem:lengthPositiveForNP}
Let $x = w\varepsilon^\mu\in \widetilde W$ and $v\in \LP(x)$. Put $J = \supp_\sigma(v^{-1}\sigma(wv))$.
\begin{enumerate}[(a)]
\item The Newton-point of $\dot x$, denoted $\nu(\dot{x})$, lies in
\begin{align*}
\nu(\dot{x}) \in W_J\cdot v^{-1}\frac 1N\sum_{k=1}^N (\sigma w)^k \mu,
\end{align*}
where $N\gg 1$ is chosen such that the action of $(\sigma w)^N$ on $X_\ast(T)_{\Gamma_0}$ is trivial.
\item Suppose that $v^{-1} \sigma(wv)$ is $\sigma$-elliptic in $W_J$, i.e.\ not $\sigma$-conjugate to any element in a proper and $\sigma$-stable parabolic subgroup. Then
\begin{align*}
\nu(\dot{x}) = \pi_J \Bigl[v^{-1} \frac 1N\sum_{k=1}^N (\sigma w)^k \mu\Bigr].
\end{align*}
%Here, $\pi_J$ is the $J$-average function from \cite[Definition~3.2]{Chai2000}.
\end{enumerate}
\end{lemma}
\begin{proof}
\begin{enumerate}[(a)]
\item One can see this by using Theorem~\ref{thm:Palcove}, realizing that the Newton-point of $\dot x$ is the Newton point of $\dot{\tilde x}$. Alternatively, pick $\alpha\in \Phi^+\setminus\Phi_J$. Then
\begin{align*}
&\langle v^{-1}\sum_{k=1}^N (\sigma w)^k (\sigma w)^k \mu,\alpha\rangle = \sum_{k=1}^N \langle \mu, (\sigma w)^k v\alpha\rangle \\=& \sum_{k=1}^N \langle \mu,(\sigma w)^k v\alpha\rangle + \Phi^+((\sigma w)^k v\alpha) - \Phi^+((\sigma w)^{k+1} v\alpha)
\\=&\sum_{k=1}^N \ell(x,(\sigma w)^k v\alpha) = \sum_{k=1}^N \ell(x,v(v^{-1} \sigma wv)^k\alpha).
\end{align*}
Note that each root $(v^{-1}\sigma wv)^k\alpha$ lies in $\Phi^+\setminus \Phi_J$, so that each term in the final sum is non-negative by length positivity of $v$.
\item Let $u\in W_J$ such that $v(x) = u v^{-1} \frac 1N \sum_{k=1}^N (\sigma w)^k\mu$ as in (a). Then $uv^{-1}\sigma(wvu^{-1})$ stabilizes $\nu(\dot{x})$. By assumption, we have $\supp_\sigma(uv^{-1}\sigma(wvu^{-1})) = J$, so $W_J$ stabilizes $\nu(\dot{x})$. So write $\nu(\dot{x}) = \pi_J(\nu(\dot{x}))$ and substitute the formula from (a) to get the result.\qedhere
\end{enumerate}
\end{proof}
\begin{lemma}\label{lem:LPsimpleReflections}
Let $x = w\varepsilon^\mu\in\widetilde W$ and $a\in \Delta_{\af}$, whose classical projection we denote by $\alpha := \cl(a)\in \Phi$.
\begin{enumerate}[(a)]
\item If $\ell(x,\alpha)>0$, then $x>xr_a$ and $s_\alpha \LP(x)\subseteq \LP(xr_a)$. If $\ell(x,\alpha)>1$, then the last two sets agree.
\item If $\ell(x,\alpha)<0$, then $x<xr_a$ and $s_\alpha\LP(x) = \LP(xr_a)$.
\item If $\ell(x,\alpha)=0$, then $x<xr_a$ and $s_\alpha \LP(x)\supseteq \LP(xr_a)$.

Let $v\in \LP(x)$. If $v^{-1}\alpha\in \Phi^-$, then $s_\alpha v\in \LP(xr_a)$. If instead $v^{-1}\alpha\in \Phi^+$, then $s_\alpha v\in \LP(x)$ and $v\in \LP(xr_a)$.
\end{enumerate}
\end{lemma}
\begin{proof}
The equivalence $x>xr_a\iff \ell(x,\alpha)>0$ follows from \cite[Lemma~2.9]{Schremmer2022_newton}. To compare the length positive sets in (a), write $x$ as a length additive product of $xr_a$ and $r_a$, then apply \cite[Lemma~2.13]{Schremmer2022_newton}. Similarly, for (b) and (c) write $xr_a$ as a length additive product of $x$ and $r_a$. The final statement in (c) is \cite[Corollary~4.7]{Schremmer2024_bruhat}.
\end{proof}
\begin{remark}
Analogous conditions for statements of the form $r_ax<x$ can be obtained from \cite[Lemma~2.12]{Schremmer2022_newton}.
\end{remark}

\section{Generic Newton points and minimal Newton points}\label{sec:3}

\subsection{Minimal Newton point of positive Coxeter type elements}

\begin{proposition}\label{prop:minNewton}
Let $(x=w\varepsilon^{\mu},v)$ be a positive Coxeter pair with support $J=\supp_{\s}(v^{-1}\s(wv))$. Then there is a unique minimal element in $B(G)_x$, whose Newton point equals
$$\nu(b_{x,\min}) = \nu(\dot{x}) = \pi_J(v^{-1}\mu).$$% Write $u = wvz$ for some $z\in W_J$.
\end{proposition}
\begin{proof}
The identity $\nu(\dot x) = \pi_J(v^{-1}\mu)$ can be seen from Lemma~\ref{lem:lengthPositiveForNP}.

Write $wv = uz$ for some $u\in W^{J}$ and $z\in W_J$. By \S~\ref{sec:Palcove}, we see that $x$ is a (normalized) $(J,u,\s)$-alcove element. Let $\tilde{x} = u^{-1}x\s(u) = \varepsilon^{z^{-1}v^{-1}\mu}z^{-1}v^{-1}\s(wv)\s(z) $. By Theorem \ref{thm:Palcove}, the minimal Newton point $\nu(b_{x,\min})$ of $x$ equals
$$\min\{\nu_{M_J}(b)\mid [b]_{M_J}\in B(M_J)_{\tilde x}\},$$
where $\nu_{M_J}$ is the $M_J$-dominant Newton point. Since $\nu(\dot{x})$ is centralized by $M_J$, we see that it must be the smallest element in the above set.
%Since $\cl(\tilde{x})$ is $\s$-conjugate to a $\s$-Coxeter element in $W_J$. Therefore, we have $\nu_{M_J}(\dot{\tilde{x}}) = \pi_J(z^{-1}v^{-1}\mu) = \pi_J( v^{-1}\mu)$, which is central over $J$. Hence $\nu_{x,\min} = \nu(\dot{x}) = \pi_J(v^{-1}\mu)$.
\end{proof}
%\remind{\cite[Theorem~1.1]{Viehmann2021} only give characterize the class $[b_{x,\min}]$ but does not calculate $\nu_{x,\min}$ --QY}

\subsection{Quantum Bruhat graph}

For any $x\in\tW$, there is a unique maximal element in $B(G)_x$, which is denoted by $[b_{x,\max}]$. We denote its Newton point by $\nu(b_{x,\max})$, called the generic Newton point of $x$. The first author gives a characterization of generic Newton point in \cite{Schremmer2022_newton} via quantum Bruhat graph, as we now explain. The notion quantum Bruhat graph will only be used in Proposition \ref{prop:finiteCoxGNP}.

The quantum Bruhat graph $\text{QBG}(\Phi)$ is a directed graph, whose set of vertices is the set $W_0$. Its edges are of the form $w\rightarrow ws_\a$ for $w\in W_0$ and $\a\in \Phi^+$ whenever one of the following conditions is satisfied:
\begin{itemize}
\item $\ell(ws_\a) = \ell(w)+1$ or
\item $\ell(ws_\a) = \ell(w)+1-\langle \a^\vee,2\rho\rangle$.
\end{itemize}
The edges satisfying the first condition are called \emph{Bruhat} edges and the edges satisfying the second condition are called \emph{quantum} edges. Note that the Bruhat edges go upwards whereas the quantum edges go downwards.

It is easy to see that $\text{QBG}(\Phi)$ is connected. For $w,w'\in W$, the QBG distance $d_{\text{QBG}}(w \Rightarrow w')\in\mathbb Z_{\ge 0}$ is defined to be the length of shortest path in $\text{QBG}(\Phi)$ from $w$ to $w'$.

The weight of a Bruhat edge is defined to be zero while the weight of a quantum edge $w\rightarrow ws_\a$ is defined to be $\a^{\vee}$. The weight of a path is defined to be sum of weights of each edge of the path.
\begin{lemma}[{\cite{Brenti1998}}]
Let $w,w'\in W$. Then all shortest paths from $w$ to $w'$ in $\text{QBG}(\Phi)$ have the same weight. We denote this weight by $\wt(w\Rightarrow w')$.\rightqed
\end{lemma}

The following is the quantum Bruhat graph of type $A_2$:

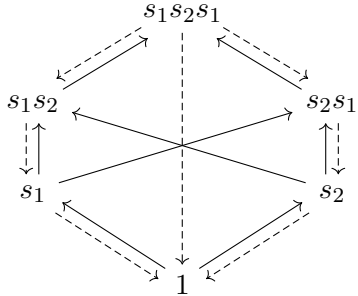
\def\qecolor{dashed}
\def\seshift{0.5ex}
\def\quantumEdge{dashed}
\def\shortEdgeShiftRight{0.5ex}
\begin{align*}
\begin{tikzcd}[ampersand replacement=\&,column sep=2em]
\&s_1 s_2 s_1\ar[ddd,\qecolor]
\ar[dl,\qecolor,shift right=\seshift]\ar[dr,\qecolor,shift left=\seshift]\\
s_1 s_2\ar[ur,shift right=\seshift]\ar[d,\qecolor,shift right=\seshift]\&\&
s_2 s_1\ar[ul,shift left=\seshift]\ar[d,\qecolor,shift left=\seshift]\\
s_1\ar[u,shift right=\seshift]\ar[urr]\ar[dr,\qecolor,shift right=\seshift]\&\&s_2\ar[u,shift left=\seshift]\ar[ull]\ar[dl,\qecolor,shift left=\seshift]\\
\&1\ar[ru,shift left=\seshift]\ar[lu,shift right=\seshift]
\end{tikzcd}
\end{align*}

For any $\l\in X_*(T)_{\G_0}$, define $\conv(\l) $ as the unique maximal element in the set
$$\{\pi_J(\l) \mid J \subseteq \BS, \s(J)=J\},$$
see \cite[\S3.1]{Schremmer2022_newton}. Note that $\avg_{\s}(\l)$, $\pi_J(\l)$, $\conv(\l)$ are also well-defined for $\l\in X_*(T)_{\G}$.

For any $[b]\in B(G)$, Hamacher and Viehmann define $\l$-invariant in \cite[Lemma/Definition 2.1]{Hamacher2018} as the unique maximal element in the set
$$\{\l\in X_*(T)_{\G} \mid   \avg_{\s}(\l)\le \nu(b)\text{ and }\k(b)=\k(\l) \text{ in }\pi_1(G)_{\G}  \}.$$
By \cite[Lemma 3.8]{Schremmer2022_newton}, we have $\conv(\l(b)) = \nu(b)$.

We have the following description of the generic Newton point $[b_{x,\max}]$.

\begin{theorem}{{\cite[Theorem 4.2]{Schremmer2022_newton}}}\label{thm:generic}
Let $[b_{x,\max}]$ the maximal element in $B(G)_x$. Then  

$$\l(b_{x,\max}) = \max_{v\in\LP(x)} (v^{-1}(\mu) - \wt(v\Rightarrow \s(wv)))\in X_*(T)_{\G}.$$
\end{theorem}

In order to apply Theorem \ref{thm:generic} to positive Coxeter type elements, we need the notion of reflection length.

%Recall $V=X_*(T)_{\G_0}\otimes\BR$. Let $\text{Aff}(V)$ the set of affine transformation on $V$. The group $  \tW \rtimes\<\s\>$ can be viewed as a subgroup of $\text{Aff}(V)$. For $x\in \tW$, we have $\s x \s^{-1} = \s(x) \in \tW \rtimes\<\s\> $.

\begin{definition}\label{def:finiteCoxeter}
For $w\in W$, we write $X_\ast(T)_{\mathbb Q}^{\sigma w}$ for the set of fixed points of the composite action $\sigma\circ w$ on $X_\ast(T)_{\Gamma_0}\otimes\mathbb Q$. We define
\begin{align*}
\ell_{R,\sigma}(w) := \dim X_\ast(T)_{\mathbb Q}^{\sigma} - \dim X_\ast(T)_{\mathbb Q}^{\sigma w}.
\end{align*}
\end{definition}

If $G$ is split, then $\ell_{R,\sigma}(w) = \ell_{R}(w)$ is the \emph{reflection length} of $w$, i.e., the minimal number $\ell$ such that $w$ can be written as product of $\ell$ reflections of $W$, cf.\ \cite{Carter2006}.

\begin{lemma}\label{lem:reflectionLength}
Let $w\in W$.
\begin{enumerate}[(a)]
\item If $w'\in W$ is $\sigma$-conjugate to $w$, then $\ell_{R,\sigma}(w) = \ell_{R,\sigma}(w')$.

\item Let $\alpha_1,\dotsc,\alpha_n\in \Phi$ such that $w = s_{\alpha_1}\cdots s_{\alpha_n}$, and let $U\subseteq X^\ast(T)_{\mathbb Q}^{\sigma}$ denote the subspace generated by their $\sigma$-averages:
\begin{align*}
U = \mathbb Q\avg_\sigma(\alpha_1)+\cdots +\mathbb Q \avg_\sigma(\alpha_n).
\end{align*}
Then $\ell_{R,\sigma}(w)\leq \dim U$.
\item The element $w$ is a partial $\sigma$-Coxeter element if and only if $\ell_{R,\sigma}(w) = \ell(w)$.
\item Among all $\sigma$-stable subsets $J\subseteq \Delta$ such that the corresponding parabolic subgroup $W_J\subseteq W$ contains a $\sigma$-conjugate of $w$, choose one that is minimal. Then $\ell_{R,\sigma}(w) = \#(J/\sigma)$, the number of $\sigma$-orbits of simple roots in $J$.
\end{enumerate}
\end{lemma}
\begin{proof}
\begin{enumerate}[(a)]
\item If $w' = v^{-1}w\s( v)$ and $p\in X_\ast(T)_{\mathbb Q}$, then
\begin{align*}
&p\in X_\ast(T)_{\mathbb Q}^{\sigma w'} \iff p = \sigma(w'(p))\iff p = \s( v)^{-1} (\sigma w)  \s(v)(p) 
\iff \s(v)(p) \in X_{\ast}(T)_{\mathbb Q}^{\sigma w}.
\end{align*}
Hence $\dim X_\ast(T)_{\mathbb Q}^{\sigma w'} = \dim X_\ast(T)_{\mathbb Q}^{\sigma w}$.
\item Define $U'\subseteq X_\ast(T)_{\mathbb Q}^\sigma$ as the orthogonal complement of $U$, i.e.\
\begin{align*}
U' = \{p\in X_\ast(T)_{\mathbb Q}^\sigma \mid \forall q\in U:~\langle p,q\rangle=0\}.
\end{align*}
Then the codimension of $U'$ in $X_\ast(T)_{\mathbb Q}^\sigma$ is at most $\dim U$. For $p\in U'$ and $i\in\{1,\dotsc,n\}$, we have $\langle p,\alpha_i\rangle = \langle p,\avg_\sigma(\alpha_i)\rangle=0$ since $p=\sigma(p)$. It follows that $U'\subseteq X_\ast(T)_{\mathbb Q}^{\sigma w}$.
\item If $\ell(w) = \ell_{R,\sigma}(w)$, we calculate
\begin{align*}
\#(\supp_\sigma(w)/\sigma)\leq \ell(w) = \ell_{R,\sigma}(w)\underset{\text{(b)}}\leq \#(\supp_\sigma(w)/\sigma).
\end{align*}
Thus $w$ is a partial $\sigma$-Coxeter element.

Assume now that $w$ is a partial $\sigma$-Coxeter element and pick a reduced word $w = s_{\alpha_1}\cdots s_{\alpha_n}$. By (b), we have $\ell_{R,\sigma}(w)\leq n$.

Let $p\in X_\ast(T)_{\mathbb Q}^{\sigma w}$. We claim that $p\in X_\ast(T)_{\mathbb Q}^\sigma$ and $\langle p,\alpha_i\rangle=0$ for $i=1,\dotsc,n$:

Indeed, if this was not the case, we would choose $i\in\{1,\dotsc,n\}$ maximally such that $\langle p,\alpha_i\rangle\neq 0$. Then we may write
\begin{align*}
w p = s_{\alpha_1}\cdots s_{\alpha_i}(p) \equiv p - \langle p,\alpha_i\rangle \alpha_i^\vee\pmod{\alpha_1^\vee,\dotsc,\alpha_{i-1}^\vee}.
\end{align*}
Since the simple roots $\alpha_1,\dotsc,\alpha_n$ have pairwise distinct $\sigma$-orbits, we get $\avg_\sigma(p)\neq \avg_\sigma(wp)$, contradicting $p\in X_\ast(T)_{\mathbb Q}^{\sigma w}$.

This argument proves $\langle p,\alpha_i\rangle=0$ for $i=1,\dotsc,n$. In particular $p = wp$. The condition $p\in X_\ast(T)_{\mathbb Q}^{\s w}$ thus implies $p=\sigma(p)$.

We conclude that $X_\ast(T)_{\mathbb Q}^{\sigma w}$ is contained in a subspace of $X_\ast(T)_{\mathbb Q}^{\sigma}$ cut out by $n$ linearly independent conditions. Thus $\ell_{R,\sigma}(w)\geq n$.
\item If $w$ is $\sigma$-conjugate to an element of $W_J$, then $\ell_{R,\sigma}(w)\leq \#(J/\sigma)$ by (b). For the converse, assume that $\ell_{R,\sigma}(w)<\#(\Delta/\sigma)$. It suffices to show that $w$ is $\sigma$-conjugate to an element of $W_J$ for some proper and $\sigma$-stable subset $J\subsetneq \Delta$, as then the claim follows from induction.

The assumption means that there is a non-zero vector $\mu \in X_\ast(T)_{\mathbb Q}^{\sigma\circ w}$. Let $v\in W$ be chosen such that $v^{-1}\mu$ is dominant. Then
\begin{align*}
\sigma^{-1}(v^{-1}\mu) = \sigma^{-1}(v^{-1} \sigma w(\mu)) = \s^{-1} (v)^{-1} wv (v^{-1}\mu)
\end{align*}
is dominant (since $\sigma$ preserves dominant coweights) and in the same $W$-orbit as $v^{-1}\mu$. Hence $v^{-1}\mu = \s^{-1}(v)^{-1} wv (v^{-1}\mu)$, proving that $w' = \s^{-1} (v)^{-1} wv$ lies in $W_J$, where $J = \{\alpha\in \Delta\mid \langle v^{-1}\mu,\alpha\rangle=0\}$ is $\sigma$-stable.
\qedhere
\end{enumerate}
\end{proof}

Recall the following conjugacy result for $\s$-Coxeter elements \cite[Theorem~7.6]{Springer1974}.
\begin{lemma}\label{lem:coxeterElementsConjugate}
Any two $\sigma$-Coxeter elements in $W$ are $\sigma$-conjugate. If $W$ is $\sigma$-irreducible, then the $\sigma$-centralizer of a $\sigma$-Coxeter element $c$ is given by
\begin{align*}
\{w\in W\mid w^{-1}c\sigma(w)= c\} = \langle c\sigma(c)\cdots \sigma^{N-1}(c)\rangle,
\end{align*}
where $N$ is the smallest integer $\geq 1$ such that the action of $\sigma^N$ on the root system $\Phi$ is trivial.\rightqed
\end{lemma}

We can now study the generic Newton point of positive Coxeter type elements.

\begin{proposition}\label{prop:finiteCoxGNP}
Let $(x=w\varepsilon^{\mu},v)$ be a positive Coxeter pair with support $J$. Then the $\l$-invariant of $[b_{x,\max}]$ is given by \begin{align*}\lambda(b_{x,\max} ) = v^{-1}\mu-\wt(v\Rightarrow\s(wv))\in X_\ast(T)_\Gamma.\end{align*}
\end{proposition}
\begin{remark}\label{rem:qbgForCoxeter}
In the context of Proposition~\ref{prop:finiteCoxGNP}, the weight function is very easy to compute: Fix a reduced word $v^{-1}\sigma(wv) = s_{\alpha_1}\cdots s_{\alpha_n}$ and note that the $\alpha_i$ are pairwise distinct. We get a path in the quantum Bruhat graph
\begin{align*}
p: v\rightarrow v s_{\alpha_1}\rightarrow \cdots \rightarrow vs_{\alpha_1}\cdots s_{\alpha_n} = \sigma(wv)
\end{align*}
of length $n$ \cite[Lemma~4.3]{Milicevic2020}. The weight of $p$ is the sum of those coroots $\alpha_i^\vee$ where the path is length decreasing, i.e.\ \begin{align*}\ell(vs_{\alpha_1}\cdots s_{\alpha_{i-1}}) > \ell(vs_{\alpha_1}\cdots s_{\alpha_i}).\end{align*}
It follows easily from Lemma~\ref{lem:reflectionLength} that $p$ is a path of minimal length from $v$ to $\sigma(wv)$.
\end{remark}

\begin{proof}[Proof of Proposition~\ref{prop:finiteCoxGNP}]
Let $v'\in \LP(x)$ such that $d_{\text{QWG}}(v'\Rightarrow\s(wv'))$ is minimal among all elements of $\LP(x)$. Then
\begin{align*}
d_{\text{QBG}}(v\Rightarrow\s(wv)) &\leq \ell(v^{-1}\s(wv))= \ell_{R,\sigma}(v^{-1}\s(wv)) \\&= \ell_{R,\sigma}((v')^{-1}\s(wv'))\leq d_{\text{QBG}}(v'\Rightarrow\s(wv')).
\end{align*}

We see that $v$ already minimizes $d(v\Rightarrow\s(wv))$. By \cite[Corollary~4.4]{Schremmer2022_newton}, this proves the claim.
\end{proof}
\begin{example}
If $x=w\varepsilon^\mu$ is of positive Coxeter type and $v\in\LP(x)$ minimizes $d_{\text{QBG}}(v\Rightarrow\s(wv))$, it is in general not to be expected that $v^{-1}\s(wv)$ is a partial $\sigma$-Coxeter element:

Indeed, consider a split group of type $A_3$ and $x=w=s_1s_2s_3\in W\subseteq \tW$. Then $(x,1)$ is a pair of finite Coxeter type. Putting $v=s_2$, we get
\begin{align*}
d_{\text{QBG}}(v\Rightarrow wv) = 3 = d_{\text{QBG}}(1\Rightarrow w),
\end{align*}
which is minimal by the above lemma, but $v^{-1} wv = s_2 s_1 s_2 s_3 s_2$ is not a Coxeter element.
\end{example}

\subsection{The support of positive Coxeter type elements}\label{sec:support}

Let $(x=wt^\mu ,v)$ be a positive Coxeter pair with support $J = \supp(v^{-1}\s(wv))$. The element $v$ and the support $J$ are not unique, i.e., one may find some different $v'\in W$ and $J'\subset \BS$ such that $(x',v')$ is positive Coxeter pair with support $J'$. It is very natural to ask to what extent the support $J$ can vary.
\begin{example}\label{eg:J}
Consider $A_5$ case and let $x = \t_3\cdot s_1$. Then $(x,v=s_3s_4s_2)$ is a positive Coxeter pair with support $J = \{1,2,3,5\}$ while $(x,v'=s_5s_2s_3s_4s_3s_1s_2)$ is a positive Coxeter pair with support $J' = \{1,3,4,5\}$.     
\end{example}

For any $[b]\in B(G)$, recall that $\l(b)$ is the $\l$-invariant of $[b]$ (see \S\ref{sec:3}). Write
$$\nu(b) -\avg_\s(\l(b)) = \sum_{\a\in\Delta} c_{\a}\a^{\vee}   ,$$
with $c_{\a}\in \BQ_{\ge0}$. Set
$$I_1( b)=\{ \a\in\Delta\mid c_{\a}\ne0 \}.$$

\begin{lemma}[{\cite[Lemma 3.8]{Schremmer2022_newton}}]\label{lem:I1JI}
Let $J$ be a $\s$-stable subset of $\BS$. Let $[b]\in B(G)$. Then the followings are equivalent.
\begin{enumerate}[(a)]
\item There exist $\l\in X_*(T)_{\G_0}$ with $\k(\l) = \k(b)$ such that $\nu(b) = \pi_J(\l)$.
\item $I_1(b)\subseteq J \subseteq I(\nu(b))$.
\end{enumerate}
\end{lemma}

\begin{proposition}\label{prop:J}
\begin{enumerate}[(a)]
    \item Let $(x ,v)$ be a positive Coxeter pair with support $J$. Then $ I_1(\dot{x})\subseteq J \subseteq I (\nu( \dot{x}) )$.    
\item Let $(x,v)$, $(x,v')$ are two positive Coxeter pairs with support $J$ and $J'$ respectively. Then there is some $n\in W_{I(\nu(\dot x))}$ with $\s(n)=n$ such that $nJn^{-1} = J'$. 
\end{enumerate}
\end{proposition}

\begin{proof}
\begin{enumerate}[(a)]
\item Apply Proposition \ref{prop:minNewton} to $(x,v)$, we get $\nu(\dot{x}) =  \pi_{J }( v^{-1}\mu)$. By Lemma \ref{lem:I1JI}, we get 
$ I_1(\dot{x})\subseteq J \subseteq I (\nu(\dot{x}) )$.

\item Write $x = w\varepsilon^{\mu}$. By assumption, $v^{-1}\s(wv)$ is a $\s$-Coxeter element of $W_J$ and $v'^{-1}\s(wv')$ is a $\s$-Coxeter element of $W_{J'}$. By \cite[Theorem~5.9]{Marquis2020}, it suffices to prove that $v^{-1}\s(wv)$ and $v'^{-1}\s(wv')$ are $\s$-conjugate in $W_{I(\nu(b))}$. Let $N$ be such that $(\s w)^N=\id \in \tW\rtimes \<\s\>$. Set
$$\nu_x = \frac{1}{N}\sum_{k=1}^{N}(\s w)^k(\mu),$$
Then Lemma~\ref{lem:lengthPositiveForNP} shows $\nu(\dot{x}) \in (W_J v^{-1}\nu_x)\cap (W_{J'} (v')^{-1}\nu_x)$. By (a), both $W_J$ and $W_{J'}$ stabilize $\nu(\dot{x})$. Thus $\nu(\dot{x})=v^{-1}\nu_x = (v')^{-1}\nu_x$.

Thus, $v'^{-1}v \in W_{I(\nu(b))}$. Then $v^{-1}\s(wv)$ and $v'^{-1}\s(wv')$ are $\s$-conjugate in $W_{I(\nu(b))}$.

\qedhere\end{enumerate}
\end{proof}

In Example \ref{eg:J}, we have $I_1(b) = \{1,3,5\} $ and $I(\nu(b)) =\{1,2,3,4,5\}$. We see that both $J$ and $J'$ are between $I_1(b)$ and $I(\nu(b))$. For any $\s$-stable subset $I\subset \BS$, denote by $w_0(I)$ the longest element of $W_I$. Let $n = w_0(\BS) \cdot w_0(J)$, then $nJn^{-1} = J'$.

\begin{remark}
In the proof of Proposition \ref{prop:J} (b), \cite[Theorem~5.9]{Marquis2020} in fact implies that there is a sequence of $\s$-stable subset $J_0 = J(b)$, $J_1,\ldots, J_{r-1}, J_r=J'(b)$ such that $J_i$ and $J_{i+1}$ are conjugate by the element $n_i = w_0(J_i\cup J_{i+1})\cdot w_0(J_i)$ and $n =  n_rn_{r-1}\cdots n_1$ which is a length additive product. This is \emph{Lusztig--Spaltenstein algorithm}, see \cite[Lemma~2.12]{Lusztig1979} and \cite[Proposition~5.5]{Deodhar1982}. We do not use this exact description in the rest of paper.
\end{remark}

%\begin{remark}
%    The condition $J\cap I(\nu(b)) \supset J_1(b)$ provides a necessary condition for positive Coxeter type element. For example, in $A_3$ case, let $x = \varepsilon^{(1,0,0,1,0,0)}s_{1245}$, then $\nu(b_x) = (1/2,1/2,1/4,1/4,1/4,1/4)$, Then $J\cap I(\nu(b_x))= \{1,4,5\}$ but $J(b_{x,\max}) = \{1,3,5\}$. Hence $x$ is not of finite Coxeter type. 
%\end{remark}

\section{Affine Deligne-Lusztig varieties of positive Coxeter type}\label{sec:4}

\subsection{Reduction step}\label{sec:reduction}
We have figured out the generic Newton points of positive Coxeter type elements in \S~\ref{sec:3}. The next step is to study the Deligne-Lusztig reduction of positive Coxeter type elements. It is proved in \cite[\S4.4]{Goertz2015} that the $(J,u,\s)$-alcove condition is preserved under Deligne-Lusztig reduction. We prove that the condition on positive Coxeter type is also preserved under Deligne-Lusztig reduction.

\begin{lemma}\label{lem:reduction1}
Let $(x = w\varepsilon^\mu,v)$ be a positive Coxeter pair. Let $a = (\alpha,k)\in \Delta_\af$ be a simple affine root such that $\ell(r_a x r_{\s (a)}) = \ell(x)$. Then there exists $v'\in W$ such that $(r_a x r_{\s (a)}, v')$ is positive Coxeter pair of support
\begin{align*}
\supp_\s(v^{-1}\s(wv)) = \supp_\s(v'^{-1}\s(s_\alpha ws_{\s\alpha} v')).
\end{align*}
\end{lemma}
\begin{proof}
Without loss of generality, $x\neq r_a x r_{\sigma a}$. Thus exactly one of the two conditions $r_ax < x$ and $xr_{\sigma a}<x$ is true. The argument is entirely symmetric, so let us assume for simplicity that $r_a x <x< xr_{\sigma a}$. We use the criteria from Lemma~\ref{lem:LPsimpleReflections} throughout this proof.

The condition $x<xr_{\sigma a}$ is equivalent to $\ell(x,\sigma\alpha)\leq 0$. If we have $v^{-1}\sigma \alpha\in\Phi^-$, we get $s_{\sigma \alpha} v\in \LP(xr_{\sigma a})\subseteq \LP(r_a x r_{\sigma a})$ \cite[Lemma~2.12]{Schremmer2022_newton}.  We can choose $v' = s_{\sigma \alpha} v$.

Hence let us assume $v^{-1}\sigma \alpha\in \Phi^+$. By length positivity of $v$ and the condition $\ell(x,\sigma \alpha)\leq 0$, we get $\ell(x,\sigma \alpha)=0$. By Lemma~\ref{lem:LPsimpleReflections}, we get $s_{\sigma\alpha} v \in \LP(x)$.

Define $\beta := v^{-1}\sigma \alpha\in \Phi^+$ and $J := \supp(s_\beta)\subseteq \Delta$. We can pick a reduced word
\begin{align*}
v^{-1}\sigma(wv) = s_{\alpha_1}\cdots s_{\alpha_\ell}
\end{align*}
such that $\alpha_1,\dotsc,\alpha_j\in J$ and $s_{\alpha_{j+1}}\cdots s_{\alpha_\ell} \in \prescript J{}{}W$.

The condition $r_ax<x$ implies $(wv)^{-1}\alpha\in \Phi^-$. Thus
\begin{align*}
\left(v^{-1}\s(wv)\right)^{-1} \beta = \sigma((wv)^{-1}\alpha)\in \Phi^-\implies
\beta = s_{\alpha_1}\cdots s_{\alpha_{i-1}}(\alpha_i)
\end{align*}
for a uniquely determined index $i\in\{1,\dotsc,\ell\}$. In fact, we get $i\leq j$ by choice of the reduced word and the index $j$. We may write
\begin{align*}
s_{\beta} = s_{\alpha_1}\cdots s_{\alpha_j}\cdots s_{\alpha_1}.
\end{align*}
Since $J = \supp(s_\beta)$, we see that the above expression must be reduced and $i=j$.

Define elements $v=v_0,\dotsc,v_i\in W$ as $v_k = v s_{\alpha_1}\cdots s_{\alpha_k}$. We claim that each $(x,v_k)$ is a positive Coxeter pair of support $\supp_\sigma(v_k^{-1}\s(wv_k)) = \supp_\sigma(v^{-1}\s(wv))$ for $k=0,\dotsc,i$. The claim for $k=0$ is clear.

In the inductive step, observe that $\ell(x,v_k \alpha_{k+1})\geq 0$ by $v_k\in \LP(x)$. Conversely,
\begin{align*}
\ell(x,v_k\alpha_{k+1}) = -\ell(x,vs_\beta s_{\alpha_1}\cdots s_{\alpha_i}\cdots s_{\alpha_{k+2}}(\alpha_{k+1}))\leq 0,
\end{align*}
since $vs_\beta\in \LP(x)$ and $s_{k+1}\cdots s_i \cdots s_1$ is reduced.

It follows that $v_{k+1}\in \LP(x)$ by \cite[Lemma~2.14]{Schremmer2022_newton}. We have
\begin{align*}
v_{k+1}^{-1}\s(wv_{k+1}) = s_{\alpha_{k+1}} \cdots s_{\alpha_1}(v^{-1}\s(wv)) \s(s_{\alpha_1}\cdots s_{\alpha_{k+1}})
=s_{\alpha_{k+2}}\cdots s_{\alpha_\ell} s_{\sigma\alpha_1}\cdots s_{\sigma\alpha_{k+1}},
\end{align*}
proving that it is a partial $\sigma$-Coxeter element with the same $\sigma$-support as $v^{-1}\s(wv)$. This finishes the induction.

We calculate
\begin{align*}
v_i^{-1}\sigma\alpha &= (vs_{\alpha_1}\cdots s_{\alpha_i})^{-1}\sigma\alpha = s_{\alpha_i}\cdots s_{\alpha_1} v^{-1}\sigma\alpha
\\&= s_{\alpha_i}\cdots s_{\alpha_1}(\beta) = -\alpha_i,
\end{align*}
which is a negative root% \remind{Original claim was indeed false, but this is what matters -- FS}
. Upon finding a length positive element whose inverse sends $\sigma\alpha$ to a negative root, we may argue as above and see that
we may choose $v' = s_{\sigma \alpha} v_i = v s_{\alpha_1}\cdots s_{\alpha_{i-1}}$.
\end{proof}

\begin{lemma}\label{lem:reduction2}
Suppose $(x = w\varepsilon^\mu,v)$ is a positive Coxeter pair. Let $a = (\alpha,k)\in \Delta_\af$ be a simple affine root such that $\ell(r_a x r_{\sigma a}) = \ell(x)-2$. Pick a reduced word
\begin{align*}
v^{-1}\s(wv) = s_{\alpha_1}\cdots s_{\alpha_\ell}.
\end{align*}
\begin{enumerate}[(a)]
\item There exists a uniquely determined index $i\in\{1,\dotsc,\ell\}$ such that
\begin{align*}
v^{-1}\s(s_\alpha wv) = s_{\alpha_1}\cdots s_{\alpha_{i-1}} s_{\alpha_{i+1}}\cdots s_{\alpha_\ell}.
\end{align*}

\item The pair $(r_a x, v)$ is a positive Coxeter pair with the same generic $\sigma$-conjugacy class $[b_{r_a x,\max}] = [b_{x,\max}]$.% Moreover, writing $J' = \supp_\sigma(v^{-1}\s(s_\alpha wv))$ and $J = \supp_\sigma(v^{-1}\s(wv))$, we have that the $J'$-points of $(x,v)$ and the $J'$-point of $(r_a x, v)$ are equal as elements in $X(J,J')$.

\item The pair $(r_a x r_{\sigma a},s_{\sigma\alpha} v)$ is a positive Coxeter pair. We have
\begin{align*}
\lambda(b_{r_a x r_{\s(a)},\max}) = \lambda(b_{x,\max}) - \alpha_i^\vee\in X_\ast(T)_{\Gamma}
\end{align*}
with $i$ as in (a).
\end{enumerate}
\end{lemma}
\begin{proof}
The condition $\ell(r_a x r_{\sigma a}) = \ell(x)-2$ implies $v\in \LP(r_a x)$ and $s_{\sigma \alpha} v\in \LP(r_a x r_{\sigma a})$ by Lemma~\ref{lem:LPsimpleReflections}.
\begin{enumerate}[(a)]

\item By length positivity, we get $v^{-1}\sigma\alpha\in \Phi^+$ and $(wv)^{-1}\alpha\in \Phi^-$. The claim is immediate.

\item By (a), we see that $v^{-1}\s(s_\alpha wv)$ is a partial $\sigma$-Coxeter element. It is always true that $x$ and $r_a x$ have the same generic $\sigma$-conjugacy class if $r_a x r_{\sigma a}<x$: This can be seen geometrically by observing that the Deligne-Lusztig reduction method from \cite{Goertz2010b} cuts $\mathcal Ix\mathcal I$ into an open and a closed subset, with the open subset reducing to $r_a x$ and the closed subset reducing to $r_a x r_{\sigma a}$. Then the open subset has the same generic $\sigma$-conjugacy class as $\mathcal Ix\mathcal I$ itself. Alternatively, a purely combinatorial proof can be found in \cite[Section~2.3]{He2021_generic}.

%It remains to compare the $J'$-points. Write $r_a x = s_\alpha w\varepsilon^{\mu + kw^{-1}\alpha^\vee}$. By definition, the image of $\sigma(wv)^{-1}\alpha^\vee = (v^{-1}\s(wv))^{-1}v^{-1} \sigma\alpha^\vee$ in $X(v^{-1}\s(wv),J')$ is zero. The $J'$-point of $x$ is the image of $\sigma v^{-1}\mu$ in that group, and the $J'$ point of $r_a x$ is the image of $\sigma v^{-1}(\mu + kw^{-1}\alpha^\vee)$, which coincides with $\sigma v^{-1}\mu$ in $X(v^{-1}\s(wv),J')$ by the above argument.
\item Since $s_{\sigma \alpha} v\in \LP(r_a x r_{\sigma a})$ and
\begin{align*}
(s_{\sigma \alpha} v)^{-1} \s(s_\alpha w s_{\sigma \alpha}s_{\sigma \alpha} v) = v^{-1}\s(wv),
\end{align*}
we see that $(r_a x r_{\sigma a},s_{\sigma\alpha} v)$ is a positive Coxeter pair. Write
\begin{align*}
r_a x r_{\sigma a} = s_\alpha w s_{\sigma \alpha}\varepsilon^{s_{\sigma\alpha}(\mu + k(w^{-1}\alpha^\vee - \sigma\alpha^\vee))}.
\end{align*}
By Proposition~\ref{prop:finiteCoxGNP}, we get
\begin{align*}
\lambda(b_{x,\max}) =& v^{-1}\mu - \wt(v\Rightarrow\s(wv)),\\
\lambda(b_{[r_a x r_{\sigma a}]}) =& v^{-1}(\mu + k(w^{-1}\alpha^\vee - \sigma\alpha^\vee)) - \wt(s_{\sigma\alpha} v\Rightarrow s_{\sigma \alpha}\s(wv)).
\end{align*}
So it suffices to show
\begin{align}\label{eq:weightComputation1}
\wt(s_{\sigma \alpha}v\Rightarrow s_{\sigma \alpha}\s(wv)) = \wt(v\Rightarrow \s(wv)) + k(\s(wv)^{-1}\sigma\alpha^\vee-v^{-1}\sigma\alpha^\vee) + \alpha_i^\vee.
\end{align}
Put $v_1 := v s_{\alpha_1}\cdots s_{\alpha_{i-1}}$ and $v_2 := v_1 s_{\alpha_i}$, so that %$v_1^{-1}\sigma\alpha=\alpha_i\in \Phi^+, v_2^{-1}\sigma\alpha=-\alpha_i \in \Phi^-$ and
\begin{align*}
\wt(v\Rightarrow\s(wv))& = \wt(v\Rightarrow v_1) + \wt(v_1\Rightarrow v_2) + \wt(v_2\Rightarrow \s(wv))\\
\wt(s_{\sigma\alpha}v\Rightarrow s_{\sigma\alpha}\s(wv))& = \wt(s_{\sigma\alpha}v\Rightarrow s_{\sigma\alpha}v_1) + \wt(s_{\sigma\alpha}v_1\Rightarrow s_{\sigma\alpha}v_2) \\&\qquad\qquad+ \wt(s_{\sigma\alpha}v_2\Rightarrow s_{\sigma\alpha}\s(wv)).
\end{align*}Observe that $v_1^{-1}\sigma\alpha = \alpha_i\in \Phi^+$ (with $v^{-1}\sigma\alpha\in \Phi^+$ as well) whereas $v_2^{-1}\sigma\alpha =-\alpha_i\in \Phi^-$ (with $\s(wv)^{-1}\alpha\in \Phi^-$ as well).

From \cite[Lemma~7.7]{Lenart2015}, we get
\begin{align}
\wt(s_{\sigma \alpha} v\Rightarrow s_{\sigma\alpha} v_1) &= \wt(v\Rightarrow v_1) + k(v_1^{-1}\sigma\alpha^\vee-v^{-1}\sigma \alpha^\vee),\label{eq:weightComputation2}\\
\wt(s_{\sigma\alpha} v_2\Rightarrow s_{\sigma\alpha} \s(wv)) &= \wt(v_2\Rightarrow \s(wv)) + k( (\s wv)^{-1}\sigma\alpha^\vee-v_2^{-1}\sigma \alpha^\vee).\label{eq:weightComputation3}
\end{align}

Using \eqref{eq:weightComputation2} and \eqref{eq:weightComputation3} to simplify \eqref{eq:weightComputation1}, it suffices to prove
\begin{align}
\wt(s_{\sigma\alpha}v_1\Rightarrow s_{\sigma \alpha} v_2) = \wt(v_1\Rightarrow v_2) + k(v_2^{-1}\sigma \alpha^\vee - v_1^{-1}\sigma\alpha^\vee)-\alpha_i^\vee.\label{eq:weightComputation4}
\end{align}
Observe that $v_1$ and $v_2$ differ only by a multiplication by a simple reflection on the right, hence $d(v_1\Rightarrow v_2)=1$ and $\wt(v_1\Rightarrow v_2) = \alpha_i^\vee \Phi^+(-v_1\alpha_i) = k\alpha_i^\vee$. The same argument allows us to compute $\wt(s_{\sigma\alpha}v_1\Rightarrow s_{\sigma\alpha} v_2) = (1-k)\alpha^\vee$. From the simple observation
\begin{align*}
k(v_2^{-1}\sigma \alpha^\vee - v_1^{-1}\sigma\alpha^\vee) = -2k\alpha_i^\vee,
\end{align*}
we get \eqref{eq:weightComputation4}. This finishes the proof.
\qedhere\end{enumerate}
\end{proof}

\subsection{Main result}\label{sec:main}
In this subsection, we give full description of properties of $X_x(b)$ for positive Coxeter type element $x$. 

We say a subset $E\subseteq B(G)$ is \emph{saturated}, or \emph{has no gap points}, if for any $[b_1]<[b_2]<[b_3]$ in $B(G)$, if $[b_1],[b_3]\in E$ then $[b_2]\in E$.
\begin{theorem}\label{thm:fctConsequences}
Let $(x=w\varepsilon^\mu,v)$ be a positive Coxeter pair with support $J=\supp_\sigma(v^{-1}\s(wv))$. Let $[b]\in B(G)$.
\begin{enumerate}[(a)]
\item We have $X_x(b)\neq\emptyset$ if and only if
\begin{align*}
v^{-1}\mu-\wt(v\Rightarrow\s(wv))-\lambda(b)\in\sum_{\alpha\in J} \mathbb Z_{\geq 0}\alpha^\vee\in X_{\ast}(T)_{\Gamma}.
\end{align*}
In particular, the set $B(G)_x$ is saturated.

Suppose that $X_x(b)\neq\emptyset$ for the remaining statements (b)--(d).
\item The affine Deligne--Lusztig variety $X_x(b)$ is equidimensional of dimension
\begin{align*}
\dim X_x(b) = \frac 12\Bigl(\ell(x)+\ell(v^{-1}\s(wv)) - \langle \nu(b),2\rho\rangle - \defect(b)\Bigr).
\end{align*}
\item Let $\mathcal T$ be a reduction tree of $x$ in the sense of \S2.2. Then there exists a unique path $\underline p$ in $\mathcal T$ such that the endpoint of $\underline p$ lies in $[b]$. In particular, $\BJ_b(F)$ acts transitively on the set of irreducible components of $X_x(b)$. Moreover, $\ell_I(\underline p ) = \sharp \bigl((J - I(\nu(b)))/\<\s\>\bigr)$ and $\ell_{II}(\underline p ) = \< \lambda(b_{x,\max}) - \lambda(b),\rho\> =\frac{1}{2}(\ell(x)-\sharp (J/\<\s\>) -\<\nu(b),2\rho\> +\de(b) ) $.

\item Let $x'$ be the endpoint of the path $\underline p$ in (c). Then there is some $v'\in W$ such that $(x',v')$ is a positive Coxeter pair with support $J(b) = J\cap I(\nu(b))$. Moreover, $\ell(x') = \<\nu(b),2\rho\> + \sharp\bigl((J(b)-I_1(b))/\<\s\>\bigr)$.

\end{enumerate}

\end{theorem}

\begin{proof}[Proof of Theorem~\ref{thm:fctConsequences}]
Induction on $\ell(x)$. If $x$ is of minimal length in its $\sigma$-conjugacy class, we know that $B(G)_x = \{[b_{x,\max}]\}$ by \cite[Theorem~3.5]{He2014}. The reduction tree consists only of the element $x$ itself. For the one element $[b] = [b_{x,\max}]\in B(G)_x$, the identity in (b) can be seen as follows:
\begin{align*}
\dim X_x(b) &= \ell(x) - \langle \nu(b),2\rho\rangle = \ell(x) - \langle \lambda(b),2\rho\rangle - \defect(b) \\&= \ell(x) - \langle v^{-1}\mu-\wt(v\Rightarrow\sigma(wv)),2\rho\rangle - \defect(b) =  d(v\Rightarrow\sigma(wv)) - \defect(b) \\&= \ell(v^{-1}\sigma(wv)) - \defect(b).
\end{align*}
In this calculation, the first identity is \cite[Lemma~3.2]{Milicevic2020}. The second identity is \cite[Proposition~3.9]{Schremmer2022_newton}. The third identity is Proposition~\ref{prop:finiteCoxGNP}. The final two identities are \cite[Lemma~4.4]{Schremmer2022_newton} and Remark~\ref{rem:qbgForCoxeter}. The identity in (b) follows from this calculation.

We next prove that $J\subseteq I(\nu(b))$, as both (c) and (d) follow immediately from this. Since $x$ has minimal length in its $\sigma$-conjugacy class, we get $\nu(b) = \nu(b_{x}) = \nu(\dot{x})$, so this inclusion is Proposition \ref{prop:J} (a).%\remind{cite Proposition \ref{prop:J}(a) is better? -- QY}

For $[b] = [b_{x,\max}]$, the condition in (a) is satisfied by Proposition~\ref{prop:finiteCoxGNP}. We still have to prove that no other $[b]\in B(G)$ satisfies the condition in (a). Indeed, any $[b]$ satisfying the condition in (a) satisfies $\lambda(b)\leq v^{-1}\mu-\wt(v\Rightarrow\sigma(wv)) = \lambda(b_{x,\max})$, so $[b]\leq [b_{x,\max}]$. On the other hand, $\nu(b)\geq \pi_J(\nu(b)) = \pi_J(\lambda(b_{x,\max})) = \pi_J(\lambda(\dot x)) = \nu(\dot{x})$.  

We have proved all claims in the case where $x$ has minimal length in its $\sigma$-conjugacy class.
Consider now the case where $x$ is \emph{not} such a minimal length element. Pick a reduction tree as in (c). This means that we are, in particular, given a sequence of elements $x= x_1,\dotsc, x_k$ such that $x_{i+1} = r_{a_i} x_i r_{\sigma a_i}$ for some simple affine root $a_i \in \Delta_\af$ and
\begin{align*}
\ell(x_1)=\cdots=\ell(x_{k-1})>\ell(x_k).
\end{align*}
Now Lemma~\ref{lem:reduction1} allows us to pass the finite Coxeter condition from $x$ to $x_{k-1}$. In other words, we may assume that $k=2$ and $\ell(r_a x r_{\sigma a})<\ell(x)$. Now we apply Lemma~\ref{lem:reduction2} and the inductive claim. Let $\alpha_1,\dotsc,\alpha_\ell$ and $i\in\{1,\dotsc,\ell\}$ as in Lemma~\ref{lem:reduction2}.

In case
\begin{align*}
\lambda(b_{x,\max}) - \lambda(b) \in \sum_{j\neq i}\mathbb Z_{\geq 0}\alpha_j^\vee \subset X_\ast(T)_{\Gamma},\tag{$\ast$}
\end{align*}
we want to show that $X_{r_a x r_{\sigma a}}(b)=\emptyset$ and $X_{r_a x}(b)\neq\emptyset$. Indeed, Lemma~\ref{lem:reduction2} shows that
\begin{align*}
    \lambda(b_{r_a x r_{\sigma a},\max}) = \lambda(b_{x,\max})-\alpha_i^\vee\underset{(\ast)}{\not\geq} \lambda(b),
\end{align*}
so that $[b]\not\leq [b_{r_a x r_{\sigma a},\max}]$. Hence $X_{r_a x r_{\sigma a}}(b)=\emptyset$.
The Deligne--Lusztig reduction method shows that $[b_{x,\max}] = [b_{r_a x,\max}]$. Moreover, Lemma~\ref{lem:reduction2} yields that $(r_a x, v)$ is a positive Coxeter pair of support $J' = J\setminus \sigma^{\mathbb Z}(\alpha_i)$. By $(\ast)$, we know that $\lambda(b_{r_a x,\max}) - \lambda(b) \in \sum_{\alpha\in J'} \mathbb Z_{\geq 0}\alpha^\vee$. Hence the inductive assumption shows $X_{r_a x}(b)\neq\emptyset$ as claimed.

Thus, any path in the reduction tree of $x$ ending with a minimal length element for $[b]$ must start with $x\rightarrow r_a x$. By the induction assumption, there exists exactly one such path, proving uniqueness in (c). Since $\ell(r_a x) = \ell(x)-1$, the statement (b) follows for $x$. In order to see (c) and (d), it remains to prove $J\cap I(\nu(b)) = J'\cap I(\nu(b))$. In other words, we have to show $\langle \nu(b),\alpha_i\rangle>0$. Applying Theorem~\ref{thm:Palcove} to the element $r_a x$, we see that $\nu(b)\leq \nu(b_{x,\max})$ and $\nu(b)\equiv \nu(b_{x,\max})\pmod{J'}$. % by Theorem~\ref{thm:Palcove} (\remind{need more word like ``applying Theorem 2.6 to $r_ax$"? -- QY}),
Hence, we get $\langle \nu(b),\alpha_i\rangle \geq \langle \nu(b_{x,\max}),\alpha_i\rangle$. So we only have
to show $\langle \nu(b_{x,\max}),\alpha_i\rangle>0$.

Suppose that instead $\langle \nu(b_{x,\max}),\alpha_i\rangle=0$. Then we get
\begin{align*}
\nu(b_{x,\max}) &= \pi_{I(\nu(b_{x,\max}))}(\lambda(b_{x,\max})) = \pi_{I(\nu(b_{x,\max}))}(\lambda(b_{x,\max})-\alpha_i^\vee)
\\&= \pi_{I(\nu(b_{x,\max}))}(\lambda(b_{x,\max})-\alpha_i^\vee)\underset{\text{L\ref{lem:reduction2}}} = \pi_{I(\nu(b_{x,\max}))}(\lambda(b_{r_a x r_{\sigma a},\max}))\\& \leq \conv(\lambda(b_{r_a x r_{\sigma a},\max})) = \nu(b_{r_a x r_{\sigma a},\max}).
\end{align*}
Thus we would get $\nu(b_{x,\max}) = \nu(b_{r_a x r_{\sigma a},\max})$, contradicting the fact $\lambda(b_{x,\max})\neq \lambda(b_{r_a x r_{\sigma a},\max})$ from Lemma~\ref{lem:reduction2}. This contradiction shows $\langle \nu(b_{x,\max}),\alpha_i\rangle>0$, finishing the first case.

Let us now consider the case
\begin{align*}
\lambda(b_{x,\max}) - \lambda(b) \in \alpha_i + \sum_{j}\mathbb Z_{\geq 0}\alpha_j^\vee \subset X_\ast(T)_{\Gamma}.
\end{align*}
We obtain $X_{r_a x r_{\sigma a}}(b)\neq\emptyset = X_{r_a x}(b)$ by induction and Lemma~\ref{lem:reduction2}, arguing just as above. One proves statements (b) -- (d) as in the previous case.

Finally, in case
\begin{align*}
\lambda(b_{x,\max}) - \lambda(b) \notin \sum_{j}\mathbb Z_{\geq 0}\alpha_j^\vee \subset X_\ast(T)_{\Gamma},
\end{align*}
we get $X_{r_a x r_{\sigma a}}(b) = \emptyset = X_{r_a x}(b)$ by induction. This finishes the induction and the proof.
\end{proof}

\subsection{Very special parahoric subgroup}\label{sec:veryspecial}
Let $H$ be a (not necessarily quasi-split) connected reductive algebraic group over $F$. Let $\sigma_H$ be its Frobenius automorphism and $\mathbb S_{\af}$ the set of simple affine reflections of $H$, using the same setup as in \S\ref{sec:preliminary}. A $\sigma_H$-stable subset $K\subseteq \mathbb S_{\af}$ is called very special (with respect to $\sigma_H$), if the parabolic subgroup $W_K$ is finite and the longest element of $W_K$ is of maximal length among all such $\sigma_H$-stable subsets of $\mathbb S_{\af}$. If $K$ is a very special subset, then we call the standard parahoric subgroup of $H(F)$ corresponding to $K$ is very special. \cite[Proposition~2.2.5]{He2021_stabilizers} proves that $K$ is very special subset if and only if the corresponding standard parahoric subgroup of $H(F)$ is of maximal volume among all the parahoric subgroups of $H(F)$.

The following result is proved in \cite[Theorem~1.1]{He2012}. Recall the notation $\mathcal O_{\min}$ from \S\ref{sec:1.3}. We write $W_a\subseteq \widetilde W$ for the affine Weyl group $W_a = W\ltimes \mathbb Z\Phi^\vee$, i.e.\ the subgroup generated by all affine reflections.
\begin{theorem}\label{thm:HY}
    Let $\tau$ be a length zero element of $\widetilde W$ and $\mathcal{O}\subset W_a\tau$ be a $W_a$-$\sigma$-conjugacy class of $\widetilde W$ such that the classical part $\cl(\CO)$ contains some $\s$-Coxeter element of $W$. Then there is a unique very special parahoric subset $K\subset \mathbb S_{\af}$ with respect to $\tau \circ \sigma$, such that $\tau^{-1} \mathcal{O}_{\min}$ equals to the set of $\text{Ad}(\tau)\circ \sigma$-Coxeter elements of $W_K$.\rightqed
\end{theorem}
\begin{remark}
    If we study instead $\sigma$-conjugacy classes in $\tW$, the corresponding subgroup $ W_K$ is unique only up to conjugation by length zero elements in $\tW$. If moreover $\mathbf G$ is absolutely simple, this subgroup $\tW_K$ is computed explicitly, in a case-by-case fashion, in \cite{He2012}. The quasi-simple case is easily reduced to absolute simple case.
\end{remark}

\begin{example}
    (1) Assume $H$ is of type $C_2$ and $\sigma_H$ is the order 2 automorphism. Then $\{0,2\}$ is a very special subset and $\{2\}$ is not. The element $\tau_2s_2 = \varepsilon^{\omega^{\vee}_2}s_{21}$ has finite Coxeter part, while the element $\tau_2s_1 = \varepsilon^{\omega^{\vee}_2}s_{2121}$ does not.

    (2) Assume $H$ is of type $E_6$ and $\sigma_H$ is the order 3 automorphism. Then $\{2,3,4,5\}$ is the unique very special subset.
\end{example}

Let $(x=w\varepsilon^\mu,v)$ be a positive Coxeter pair with support $J=\supp_\sigma(v^{-1}\s(wv))$. Let $[b]\in B(G)_x$. We can now give a more specific description of the geometry of $X_x(b)$.

Let $\CT$ be a reduction tree of $x$ and let $x' = w' \varepsilon^{\mu'}$ be the unique endpoint of $\CT$ such that $x'\in [b]$ as in Theorem \ref{thm:fctConsequences} (c). By \S\ref{sec:1.4} and Theorem \ref{thm:fctConsequences}, we get that $X_x(b) $ is an iterated fibration over $X_{x'}(b)$ of type $(\ell_I(\underline p ) , \ell_{II}(\underline p ))$. Let $v'\in W$ such that $(x',v')$ is a positive Coxeter pair with support $J(b) = J\cap I(\nu(b))$. Let $u'\in W^{J(b)}$ such that $w'v' \in u'W_{J(b)} $. By \S\ref{sec:Palcove}, we see that $x'$ is a normalized $(J(b),u',\s)$-alcove element. Let $\widetilde{x}' = u'^{-1}x'\s(u')$ and write $M=M_{J(b)}$. Direct computation shows that $u'^{-1}w\s(u')$ is $W_{J(b)}$-$\s$-conjugate to the $\s$-Coxeter element $v'^{-1}\s(w'v')$ of $W_{J(b)}$. By \cite[Proposition~4.5]{He2015b}, $\widetilde{x}'$ is of minimal length element in its $\sigma$-conjugacy class in $\widetilde{W}_{J(b)}$. Hence by Theorem \ref{thm:HY}, there is a length zero element $\t$ of $\tW_{J(b)}$ and a very special parahoric subset $K$ of $J(b)_{\af}$ with respect to $ \tau \circ \sigma$, and a $\tau\circ \s$-Coxeter element $c_K$, such that $\widetilde{x}' =c_K\cdot \tau $.

By Theorem \ref{thm:Palcove}, the natural map $B(M)_{\tilde{x}'}\rightarrow B(G)_{x'}$ is bijective and for any $[b']\in B(M)_{\tilde{x}'} $, the $M$-dominant Newton point $\nu_{M}(b')$ coincides with $\nu(b')$. Therefore, we may assume $[b]\in  M (\breve F)$.

By \cite[Theorem 2.1.4 (b)]{Goertz2010} and \cite[Theorem 3.3.1]{Goertz2015}, we have 
$$\mathbb{J}_b(F) \backslash X_{x'}(b)\cong \mathbb {J}^{M}_b(F)\backslash X^{M}_{\widetilde{x}'}(b) \cong \mathbb {J}^{M}_{\dot{\t}}(F)\backslash X^{M}_{\widetilde{x}'}(\dot{\t}).$$
Following \cite[Theorem 4.8]{He2014}, we get that 
$$X^{M}_{\widetilde{x}'}(\dot{\t}) \cong \mathbb{J}_{\dot{\t}}^{M}(F) \times^{\mathbb{J}_{\dot{\t}}^{M}(F)\cap \CP }X',$$
where $\CP$ is the standard parahoric subgroup of $M(\breve F)$ corresponding to $K$ and then $\mathbb{J}_{\dot{\t}}^{M}(F)\cap \CP$ is the very special parahoric subgroup of $\mathbb{J}_{\dot{\t}} ^{M}(F)$ corresponding to $K$, and 
$$X' = \{ g\in \CP/\CI \mid g^{-1}\dot{\t}\s(g)\in \CI c_K\t\CI\}.$$
Note that $X'$ is isomorphic to 
$$X_{c_K}^{\overline{\CP}} = \{  \overline{g}\in \overline{\CP}/\overline{\CI} \mid \overline{g}^{-1} (\Ad(\dot{\t})\circ\s)(\overline{g})\in \overline{\CI} c_K \overline{\CI} \},$$
the classical Deligne-Lusztig variety in the reductive quotient $\overline{\CP}$ of $\CP$ corresponding to $c_K\in W_K$. Hence we get
$$\mathbb{J}_b(F) \backslash X_{x'}(b)\cong \mathbb{J}_{\dot{\t}}^{M}(F)\cap \CP  \backslash X'\cong \overline{\CP}^{\Ad(\dot{\t})\s}\backslash X^{\overline{\CP}}_{c_K}.$$

Finally, we point out that the element $\t\in \tW_{J(b)}$ is uniquely determined (up to $\tW_{J(b)}$-$\s$-conjugacy) by the property 
\begin{itemize}
\item $\k(\t)=\k(b)$;
\item the $M$-dominant Newton point $\nu^{M}(\dot{\t})$ equal $\nu(b)$.
\end{itemize}

In summary, we have proven the following

\begin{corollary}\label{cor:geo}
%$\mathbb{J}_b(F)$-equivariant universally homeomorphic to an iterated fibration of type $(\ell_I(\underline p), \ell_{II}(\underline p))$ 
Let $(x=wt^\mu,v)$ be a positive Coxeter pair with support $J$. Let $[b]\in B(G)_x$. Let $J(b)=J\cap I(\nu(b))$ and $M = M_{J(b)}$. Then there is a unique (up to $\tW_{J(b)}$-$\s$-conjugacy) length zero element $\t$ of $\tW_{J(b)}$ such that $\k(\dot{\t}) = \k(b)$ and $\nu_{M}(\dot{\t}) =\nu(b)$.

Moreover, there is a very special parahoric subset $K$ of $J(b)_{\af}$ with respect to $\t\circ\s$, such that 
$\mathbb{J}_b(F) \backslash X_x(b)$ is universally homeomorphic to an iterated fibration (of type $(\ell_I , \ell_{II} )$ as in Theorem \ref{thm:fctConsequences} (c)) over $\overline{\CP}^{\Ad(\dot\t)\s}\backslash X^{\overline{\CP}}_{c_K}$. Here, $\CP$ is the standard parahoric subgroup of $M(\breve F)$ corresponding to $K$ and $\overline{\CP} $ its reductive quotient, and $X^{\overline{\CP}}_{c_K}$ a classical Deligne-Lusztig variety in ${\overline{\CP}}$ of Coxeter type. \rightqed
\end{corollary}
Finally, we point out that $\overline{\CP}^{\Ad(\dot\t)\s}\backslash X^{\overline{\CP}}_{c_K}$ is quasi-isomorphic to the orbit space of a $\sharp K/\<\t\s\>$-dimensional affine space by an action of a finite torus (see \cite[4.3 (a)]{HL12}).

\subsection{Minimal length case}\label{sec:minimallengthcase}

Let $x = w\varepsilon^{\mu}\in\tW$. It is very natural to ask, when is $x$ a positive Coxeter type element? This is a difficult problem in general. One obvious necessary condition is that $w$ must be $\s$-conjugate to a partial $\s$-Coxeter element.
We have the following result for minimal length elements.
\begin{theorem}\label{thm:coxeterConjugacyClasses}
Let $x=w\varepsilon^\mu\in\tW$ be of minimal length in its $\sigma$-conjugacy class. Then the following are equivalent:
\begin{enumerate}[(a)]
\item The element $x$ is of positive Coxeter type.
\item The element $w$ is $\sigma$-conjugate, inside $W$, to a partial $\sigma$-Coxeter element. 
\end{enumerate}
\end{theorem}

For the proof, we need a bit of preparation.
\begin{lemma}\label{lem:conjugateToCoxeter}
    Let $J\subseteq \Delta$ be a $\sigma$-stable subset and let $w\in W_J$. Suppose that $w$ is $\sigma$-conjugate \emph{inside $W$} to a partial $\sigma$-Coxeter element of $W$. Then $w$ is also $\sigma$-conjugate inside $W_J$ to a partial $\sigma$-Coxeter element of $W_J$.
\end{lemma}
\begin{proof}
	Denote by $\mathcal O$ the $\sigma$-conjugacy class of $w$ inside $W$. We find some $w'\in\mathcal O_{\min}$ that can be reached from $w$ by a weakly length-decreasing sequence of cylic shifts, i.e.\ $w\rightarrow_\sigma w'$ (notation as in \S\ref{sec:1.3}). Since $w\in W_J$, it follows from definition of the relation $\rightarrow_\sigma$ that also $w'\in W_J$.
	
	In other words, we may assume without loss of generality that $w$ has minimal length in $\mathcal O$.  
    Now we apply the Lusztig-Spaltenstein algorithm mentioned above, in a $\sigma$-twisted version: Let $c\in W$ be a partial $\sigma$-Coxeter element which is $W$-$\sigma$-conjugate to $w$ and $J' = \supp_\sigma(c)$. Then also $c\in \mathcal O_{\min}$. By \cite[Proposition~A]{Marquis2020}, there exists some $u\in W$ with $u(J') = J''$ and $\s (u) = u$, where $J'' = \supp_\sigma(w)\subseteq J$. Thus $u c(\s (u))^{-1} = ucu^{-1}\in W_{J''}$ is a partial $\s$-Coxeter element inside $W_{J''}$. Now Marquis' result moreover implies that $w$ and $c' = u cu^{-1}\in \mathcal O_{\min}$ are related by length-preserving cyclic shifts, i.e.\ $w\approx_\sigma c'$. Hence $w$ is a partial $\sigma$-Coxeter element in $W_J$. 
\end{proof}

The following Lemma relates positive Coxeter type elements in $\tW$ to positive Coxeter type elements in $\tW_J$.
\begin{lemma}\label{lem:Palcove}
Let $J'\subseteq J$ be two $\s$-stable subsets of $\BS$ and let $x=w\varepsilon^{\mu}$. 
\begin{enumerate}[(a)]
\item Let $u\in W^J$ and $\tilde v\in W_J$. Assume that $x$ is a normalized $(J,u,\s)$-alcove element. Let $\tilde x = u^{-1}x\s(u)$. Assume that $(\tilde x,\tilde v)$ is a positive Coxeter pair for $M_J$ with support $J'$. Then $(x,\s(u)\tilde v)$ is a positive Coxeter pair for $G$ with support $J'$.
\item Let $v\in W$. Assume that $(x,v)$ is a positive Coxeter pair for $G$ with support $J'$. Let $u\in W^J$ be the minimal representative of the coset $wvW_J$. Let $\tilde x = u^{-1}x\s(u) $ and $\tilde v = \s(u^{-1})v$. Then $\tilde v \in W_J$ and $(\tilde x, \tilde v)$ is a positive Coxeter pair for $M_J$ with support $J'$. 

\end{enumerate}
\end{lemma}
\begin{proof}
\begin{enumerate}[(a)]
\item By definition of length functional, one can check directly that for any $\a\in \Phi^+_J$,
\begin{align*}
\ell( x,\s(u)\tilde v\a) &= -\Phi^+( w\s(u)\tilde v\a) +\<\mu,\s(u)\tilde v\a\> +\Phi^{+}( \s(u)\tilde v\a)\\
                         &= -\Phi^+( u^{-1}w\s(u)\tilde v\a) +\<\mu,\s(u)\tilde v\a\> +\Phi^{+}( \tilde v\a)\\
                         &= \ell(\tilde x,\tilde v \a)\ge0,
\end{align*}
where the second equality uses the fact that $u^{-1}w\s(u)\in W_J$, $\tilde v\in W_J$ and $u\in W^J$. Let $\b\in \Phi^+-\Phi_J$, set $\b' = u^{-1}w\s(u)\tilde v\b \in \Phi^+-\Phi_J$. We have
$$\ell( x,\s(u)\tilde v\a) = \ell(x,w^{-1}u\b' ),$$
which is $\ge0$ since $x$ is $(J,u,\s)$-alcove element.
Moreover, we have $(\s(u)\tilde v)^{-1} \s(x\s(u)\tilde v) = (\tilde v)^{-1}\s(\tilde x\tilde v) $, and the finite part of which is a $\s$-Coxeter element of $W_{J'}$.

\item Since $(x,v)$ is a positive Coxeter pair with support $J'\subseteq J$, we see that $x$ is a (normalized) $(J,u,\s)$-alcove element. Then $\tilde v = \s(u^{-1})v \in W_J\cdot \s(v^{-1}w^{-1})v\in W_J$. By the same calculation as in (a), one can check that for any $\a\in \Phi^+_J$,
$$\ell(\tilde x,\tilde v \a)= \ell(x,  \s(u)\tilde v \a) = \ell(x,    v \a)\ge0.$$
Moreover, $(\tilde v)^{-1}\s(\tilde x\tilde v) = (\s(u)\tilde v)^{-1} \s(x \s(u)\tilde v) =  v ^{-1} \s(x v)$, and the finite part of which is a $\s$-Coxeter element of $W_{J'}$. 
\qedhere\end{enumerate}
\end{proof}

The following lemma is of independent interest.
\begin{lemma}\label{lem:minimalElements}
    Let $x=w\varepsilon^\mu\in\tW$ be of minimal length in its $\sigma$-conjugacy class. Then the following are equivalent:
    \begin{enumerate}[(a)]
    \item The element $x$ is \emph{not} a $(J,v,\delta)$-alcove element for any $J\subsetneq \Delta$.
    \item We may write $x = \tau u$ with the following properties: The element $\tau\in\Omega$ is of length zero. There exists a $\sigma\circ\tau$-stable subset of simple roots $K\subset \Delta_{\af}$ that is \emph{spherical}, i.e.\ the subgroup $W_K$ of $\tW$ generated by the corresponding simple affine reflections is finite. We have $u\in W_K$ with $K = \supp_{\sigma\circ\tau}(u) = \supp_\sigma(x)$ and $K$ is maximal among all  $\sigma\circ\tau$-stable subsets of $\Delta_{\af}$.
	\item The element $w\in W$ is \emph{$\sigma$-elliptic}, i.e.\ not $\sigma$-conjugate to an element in some parabolic subgroup $W_J\subset W$ for a proper and $\sigma$-stable subset $J\subsetneq \Delta$.
    \end{enumerate}
\end{lemma}
\begin{proof}
By \cite[Theorem~4.8]{He2014}, the Iwahori double coset $\mathcal Ix\mathcal I$ is contained in the $\sigma$-conjugacy class $[\dot x]\in B(G)$. In the case of (a), we know that $[\dot x]$ is basic by \cite[Theorem~A]{Goertz2015}. By \cite[Proposition~5.6]{Goertz2019}, we know that $[\dot x]$ is basic if and only if we have $x = \tau u$ for some length zero element $\tau$ and $u\in W_K$, where $K\subset \Delta_{\af}$ is spherical and $\sigma\circ\tau$-stable.

We summarize that both (a) and (b) are false if $[\dot x]$ is not basic. Moreover, (c) immediately implies (a). Let us thus assume that $[\dot x]$ is basic for the remainder of the proof. Then we know that $K = \supp_\sigma(x) = \supp_{\sigma\circ\tau}(u)$ is spherical, and we are given a decomposition $x = \tau u$ with $\tau\in \Omega$ and $u\in W_K$. We have to show that (a) implies (b), and that (b) implies (c).

Observe that it suffices to prove the lemma for each $\sigma$-connected component of the Dynkin diagram of $\Delta$, i.e.\ we may assume that $\Delta$ is $\sigma$-connected. Pick a connected component $C\subseteq \Delta$ and write $\Delta = C \sqcup \sigma(C)\sqcup\cdots \sqcup \sigma^n(C)$. Write $C_{\af}$ for the corresponding connected component of $\Delta_{\af}$, and let $\tau_0,\dotsc,\tau_n$ be the length zero elements corresponding to the action of $\tau$ on $\sigma^0(C),\dotsc,\sigma^n(C)$.

If $\tau'\in\Omega$ is any length zero element and $x' = (\tau')^{-1} x\sigma(\tau')\in\tW$ is the corresponding $\sigma$-conjugate of $x$, then we can observe that the truths of the statements (a), (b), (c) does not change if we pass between $x$ and $x'$. By choosing this $\tau'$ carefully in terms of $\tau_0,\dotsc,\tau_n$, we can make sure that $(\tau')^{-1} \tau \sigma(\tau')$ acts trivially on $\sigma(C_{\af})\sqcup\cdots\sqcup \sigma^n(C_{\af})$. In other words, we may assume that $\tau_0=\cdots=\tau_{n-1}=1$.

First assume that (a) holds true but (b) does not. Then we find two $\tau\circ\sigma$-orbits $o_1, o_2\subset \Delta_{\af}$ that are not contained in $K$, but that are contained in the same $\sigma$-connected component of $\Delta_{\af}$.

Define coweights $\lambda_1,\lambda_2\in \mathbb Q\Phi^\vee$ so that for any simple root $\alpha\in\Delta$, we have
\begin{align*}
\langle \lambda_i,\alpha\rangle = \begin{cases}1,&\alpha\in o_i,\\
0,&\alpha\notin o_i,\end{cases}\quad i=1,2.
\end{align*}
Let $\theta\in \Phi$ be the longest root of the connected component $C_{\af}\cap \Delta$. Inspecting the $\sigma\circ\tau$-action on $\Delta_{\af}$, with $\tau$ being trivial on $\sigma^0(C_\af),\dotsc,\sigma^{n-1}(C_\af)$, one sees that both $\lambda_1$ and $\lambda_2$ are constant on
\begin{align*}
\{\sigma^k(\theta)\mid k\in\mathbb Z\} = \{\theta,(\sigma\cl(\tau))(\theta),\dotsc,(\sigma\cl(\tau))^n(\theta)\}.
\end{align*}

We claim that we can find integers $q_1, q_2>0$ such that $\lambda := q_1 \lambda_1 - q_2 \lambda_2\in\mathbb Z\Phi$ is $\sigma\circ\cl(\tau)$-stable and satisfies $\langle \lambda,\cl(a)\rangle=0$ for all $a\in \Delta_\af\setminus (o_1\sqcup o_2)$:
\begin{itemize}
\item Consider first the case where $o_1\sqcup o_2\subset \Delta_\af$ does not contain the non-spherical affine root $(-\theta,1)$. In this case, it suffices to choose $q_1, q_2>0$ such that
\begin{align*}
\langle \lambda,\theta\rangle=0,
\end{align*}
which is certainly possible. Then $\lambda$ pairs trivially with all $\sigma$-conjugates of $\theta$ (by the above remark on $\lambda_1,\lambda_2$ being constant on this set). It follows that $\langle \lambda,\cl(a)\rangle=0$ for all simple affine roots $a\in \Delta_\af\setminus( o_1\sqcup o_2)$. Now it is clear that $\lambda$ is $\sigma\circ\cl(\tau)$-stable.
\item Consider next the case where the $\sigma\circ\tau$-orbit of $(-\theta,1)\in\Delta_\af$ contains a spherical affine root, and moreover agrees with $o_1$ or $o_2$. In this case, assume without loss of generality that it agrees with $o_1$, and pick a simple root $a\in \Delta\cap C_{\af}\cap o_1$. Choose $q_1, q_2>0$ such that
\begin{align*}
\langle \lambda,\theta-\alpha\rangle=0.
\end{align*}
Since $\lambda$ is, by construction, constant on the set $o_1\cap \Delta$, and by the above remark constant on the $\sigma$-orbit of $\theta$, it follows that $\lambda$ is constant on the set $\{\cl(a)\mid a\in o_1\}$. Certainly, $\lambda$ is also constant on the set $o_2\subset \Delta$, and constantly equal to zero on the set $\Delta_\af\setminus (o_1\sqcup o_2)\subseteq \Delta$.
\item Consider now the case where the $\sigma\circ\tau$-orbit of $(-\theta,1)$ in $\Delta_\af$ agrees with $o_1$, and only consists of the non-spherical simple affine roots $(-\sigma^k(\theta),1)$ for $k\in\mathbb Z$. In this case, observe that $\lambda_1=0$. Choose $q_1>0$ arbitrary and let $q_2>0$ such that $\lambda\in\mathbb Z\Phi^\vee$. Then all assertions are easily verified.
\end{itemize}
The condition that $\lambda$ pairs trivially with all $\cl(a)$ for $a\in K$ can be reformulated as $\ell(\varepsilon^\lambda,\cl(a))=0$ for all $a\in K$. By \cite[Lemma~2.9]{Schremmer2022_newton}, we get that $\ell(r_a \varepsilon^\lambda)>\ell(\varepsilon^\lambda)$ for all $a\in K$. By the standard theory of Coxeter groups, it follows that $\ell(u \varepsilon^\lambda) = \ell(u) + \ell(\varepsilon^\lambda)$ as $u\in W_K$. By \cite[Lemma~2.13]{Schremmer2022_newton}, there exists an element $v\in \LP(x)\cap\LP(\varepsilon^\lambda)$. This means $v$ is length positive for $x$ and $v^{-1}\lambda$ is dominant. We compute
\begin{align*}
(v^{-1}\sigma( wv))(\sigma v^{-1}\lambda) = v^{-1} \sigma \cl(\tau)u\lambda = v^{-1}\sigma \cl(\tau)\lambda = v^{-1}\lambda.
\end{align*}
This is dominant and in the same $W$-orbit as $\sigma v^{-1}\lambda$. Thus $v^{-1}\lambda$-is stable under both $\sigma$ and $v^{-1} \sigma(wv)$. Since $\lambda\neq 0$, this shows that shows that $\supp_\sigma(v^{-1}\s(wv))\subseteq \Stab(v^{-1}\lambda)\neq \Delta$, contradicting (a).

Finally, we prove that (b) implies (c). So assume that $K\subset \Delta_{\af}$ is maximal among all $\sigma\circ\tau$-stable spherical subsets. Let $Y\subseteq X_\ast(T)_{\mathbb Q}$ be the linear subspace spanned by the coroots $\cl(a)^\vee\in \Phi^\vee$ for the affine roots $a\in K$. Thus $\dim Y = \# K$, and the automorphisms $\sigma \cl(\tau)$ and $\cl(u)$ preserve the subset $Y\subset X_\ast(T)_{\mathbb Q}$.

Assume that $w$ is not $\sigma$-elliptic. Then, by Lemma~\ref{lem:reflectionLength}, we find a non-zero vector $\lambda\in X_\ast(T)_{\mathbb Q}^{\sigma w}$. This means that $\sigma \cl(\tau)\cl(u)\lambda = \lambda$. If $a\in \Delta_{\af}\setminus K$ is any simple affine root not in $K$, then
\begin{align*}
\langle \lambda, \cl(a)\rangle = \langle \lambda, \sigma \cl(\tau)\cl(u)\cl(a)\rangle = \langle \lambda,\cl(\sigma \tau a)\rangle.
\end{align*}
Since the set $\Delta_{\af}\setminus K$ consists of only one $\sigma\circ\tau$-orbit, the coweight $\lambda$ is constant on $\{\cl(a)\mid a\in \Delta_{\af}\setminus K\}$. Expressing the longest root of each connected component $C,\dotsc,\sigma^n(C)$ as a sum of simple roots, we get straightforward linear relations that the values $(\langle \lambda,\cl(a)\rangle)_{a\in \Delta_{\af}}$ must satisfy. These imply that if $\lambda\neq 0$, there has to be a root $a\in K$ with $\langle \lambda,\cl(a)\rangle\neq 0$. So the uniquely determined coweight $\lambda'\in Y$ with $\langle \lambda,\cl(a)\rangle = \langle \lambda',\cl(a)\rangle$ for all $a\in K$ is non-zero. It is also in $Y^{\sigma \cl(\tau)\cl(u)}$.

Note that $\{\cl(a)\mid a\in K\}\subseteq \BZ\Phi^\vee $ forms the basis of a root system (with Cartan matrix given by the submatrix of the affine Cartan matrix determined by $K$). The automorphism $\sigma\cl(\tau)$ acts on it, and the element $\cl(u)$ lies in the corresponding Weyl group $W_K$. By Lemma~\ref{lem:reflectionLength} applied to this Weyl group and this automorphism, the existence of $0\neq\lambda' \in Y^{\sigma \cl(\tau)\cl(u)}$ implies that $\cl(u)$ is not $\sigma \cl(\tau)$-elliptic in $W_K$. The automorphism of Coxeter groups $W_K\rightarrow W_K, z\mapsto \cl(z)$ coming from comparing the aforementioned Cartan matrices shows that $u\in W_K$ is not $\sigma\circ\tau$-elliptic. This contradicts the fact that $u$ has minimal length in its $\sigma\circ\tau$-conjugacy class in $W_K$ (by assumption on $x$) and $\supp_{\sigma\circ\tau}(u)=K$, by a version of the aforementioned Lusztig--Spaltenstein algorithm \cite[Proposition~A]{Marquis2020}.
\end{proof}

We can now prove our theorem.
\begin{proof}[Proof of Theorem~\ref{thm:coxeterConjugacyClasses}]
The direction (a) $\implies$ (b) follows directly from the definition. Let us prove the converse direction.

The first step is passing to a minimal Levi subgroup. Let $J$ be a minimal $\s$-stable subset such that $x$ is a (normalized) $(J,u,\s)$-alcove element for some $u\in W^J$. Let $\tilde x = u^{-1}x\s(u)$. Then $\cl(\tilde x)$ satisfies condition (a) of Lemma \ref{lem:minimalElements} and hence is $\s$-elliptic by condition (c). Then by Lemma \ref{lem:conjugateToCoxeter}, $\cl(\tilde x)$ is $W_J$-$\s$-conjugate to a $\s$-Coxeter element of $W_J$. By \cite[Proposition~4.5]{He2015b}, $\tilde x$ is of minimal length in its $\s$-conjugacy class in $\tW_J$. By Lemma \ref{lem:Palcove}, it suffices to prove that $\tilde x$ is a positive Coxeter type element for $M_J$. Hence we may assume that $J=\Delta$ and $x = \tilde x$.

Let $x' = c'\varepsilon^{\mu'}$, where $c'$ is a $\s$-Coxeter element, $\mu'\in \mu+\BZ\Phi$ and $\<\mu',\a\>$ is sufficiently large for all $\a\in \Delta$. Then $x'$ is a positive Coxeter type element (see Remark \ref{rmk:HNY}). Since $x$ is of minimal length in its $\sigma$-conjugacy class, by \cite[Theorem~A~(1)]{He2014b} and \cite[Proposition~A]{Marquis2020}, we have $x'\rightarrow_{\s}x$, (see \S\ref{sec:1.3} for the notation $\rightarrow_{\s}$). Hence by Lemma \ref{lem:reduction1} and \ref{lem:reduction2}, $x$ is a positive Coxeter type element.
\end{proof}

Let $\tilde \CO$ be a $\s$-conjugacy class of $\tW$. Theorem \ref{thm:coxeterConjugacyClasses} and its proof show that the following are equivalent:
\begin{enumerate}[(i)]
\item The class $\mathcal O$ contains some positive Coxeter type element.
\item The image of $\mathcal O$ under the natural projection $\widetilde W\twoheadrightarrow W$ contains a $\s$-Coxeter element
\item All minimal length elements in $\mathcal O$ are of positive Coxeter type.
\end{enumerate}
If one instead considers the more restrictive notion of \emph{finite Coxeter part} as in \cite{He2022}, then we still get the trivial implications (iii) $\implies$ (i) $\implies$ (ii), but the converses need not hold true. See the following two examples.

\begin{example}
If $G$ is split of type $B_2$, an the image of $\mathcal O$ in $W$ is of Coxeter type, then $\mathcal O$ itself might not contain any element with finite Coxeter part. Write $\alpha$ for the short simple root and $\beta$ for the long one, and consider the straight element $x = s_\beta s_\alpha s_\beta\varepsilon^{-\lambda}$, where $\langle \lambda,\alpha\rangle=0$ and $\langle \lambda,\beta\rangle=1$. The only elements $v\in W$ such that $v^{-1} s_\beta s_\alpha s_\beta v$ is a partial Coxeter element are $v=s_\beta, s_\beta s_\alpha, s_\beta s_\alpha, s_\alpha s_\beta s_\alpha$. In each case, we get that $\langle v^{-1}\lambda,\alpha\rangle = - \langle v^{-1}\lambda,\beta\rangle=\pm 1$ and $v^{-1} w v = s_\alpha$. Thus $v^{-1}\lambda + z\alpha^\vee$ is never dominant for any $z\in\mathbb Z$. This shows that
\begin{align*}
\mathcal O = \{y^{-1} x y\mid y\in\widetilde W\}
\end{align*}
does not contain an element of the form $s_\alpha\varepsilon^\mu$ with $\mu$ dominant. In particular, $\mathcal O$ does not contain any element of finite Coxeter part in the sense of \cite{He2022}.
    
\end{example}

\begin{example}
If $G = \GL_6$ is split of type $A_5$, it is possible to construct a conjugacy class $\mathcal O\subset \widetilde W$ containing an element of finite Coxeter part, but no minimal length element of finite Coxeter part. E.g.\ we find a straight and superbasic element $\tau\in\widetilde W$ which does not have finite Coxeter part, and then its conjugacy class $\mathcal O$ will contain elements of finite Coxeter part by the proof of Theorem~\ref{thm:coxeterConjugacyClasses}. However, $\tau$ is the only minimal length element in $\mathcal O$.

\end{example}

For $x\in W_a$, we denote by $\supp x$ the support of $x$, i.e., the set of $\a\in \D$ that $s_\a$ appears in some (or equivalently, any) reduced expression of $x$. For a length zero element $\t$, define
$$\supp_{\s}(\t x) = \bigcup_{n\in\BZ}(\s\circ\t)^n ( \supp(x)).$$
%Recall that a subset $K$ of $\D_{\af}$ is called spherical if $\sharp W_K<\infty$.

The next proposition shows that if a positive Coxeter type element $x$ is not of minimal length, then its support $J$ is large.
\begin{proposition}\label{prop:pctHasLargeSupport}
Let $(x=w\varepsilon^{\mu},v)$ be a positive Coxeter pair with support $J$. Suppose $J$ does not contain any $\s$-connected component of $I(\nu(\dot{x}))$. Then $x$ is of minimal length in its $\sigma$-conjugacy class.
\end{proposition}

\begin{proof}Denote $I = I(\nu(\dot{x}))$ for simplicity. We have $I\supseteq J$ (Proposition \ref{prop:J}). By Theorem \ref{thm:fctConsequences} (c), it suffices to prove that $B(G)_x$ contains only one element. Let $u$ be the minimal representative of the coset $wvW_{I }$, then $x$ is a $(I ,u,\s)$-alcove element. Let $\tilde{x} = u^{-1}x\s(u)\in \tW_I$. Using Lemma \ref{lem:Palcove}, we get that $\tilde x$ is a positive Coxeter type element for $M_I$ (with support $J$). If we can prove that $\tilde x$ is of minimal length in its $\tW_I$-$\s$-conjugacy class, then $B(M_I)_{\tilde x}$ contains only one element, and so is $B(G)_x$ (Theorem \ref{thm:Palcove}). Hence we may assume $I=\D$ and $x=\tilde x$. Note that the proposition can be proved separately on each $\sigma$--connected component of $\Delta$, e.g.\ the element $x$ has minimal length in its $\sigma$--conjugacy class if and only if this property is satisfied under every projections to a Levi subgroups of $G$ defined by a $\sigma$--connected component of $G$. Hence we also assume that $\D$ is $\s$-connected.

Assume $x$ is not of minimal length in its $\s$-conjugacy class, one can find a sequence of elements $x = x_0\xrightarrow{s_1}_\s x_1\xrightarrow{s_2}_\s \cdots \xrightarrow{s_{n-1}}_\s x_n$ such that $x_n$ is of minimal length in its $\s$-conjugacy class and $x_{n-1}$ is not. Then $x_n$ is contained in the basic $\s$-conjugacy class $[b_0]$. By \cite[Proposition~5.6]{Goertz2019}, $\supp_{\s} (x_n)$ is spherical. Note that $\supp_{\s} (x_{n-1}) = \supp_{\s} (x_n) \cup \{ (\s\circ\t)^n( \Ad(\t^{-1})(s_{n-1})    ) \mid n\in\BZ \} $.  Since $J\subsetneqq \D$, by Lemma \ref{lem:minimalElements}, $\supp_{\s} (x_n)$ is not maximal. Since $\BS$ is $\s$-connected, we see that $\supp_{\s} (x_{n-1})/ \< \s\t\>$ is also spherical. Hence $x_{n-1}$ is contained in the basic $\s$-conjugacy class $[b_0]$ but is not of minimal length in its $\s$-conjugacy class. This violates the strong multiplicity one property (Theorem \ref{thm:fctConsequences} (c)). This completes the proof.
\end{proof}

\subsection{A converse to Theorem \ref{thm:fctConsequences}}
\label{sec:converse}
The main result of this section is the following.
\begin{proposition}\label{prop:pctCharacterization}
Let $x = w\varepsilon^\mu$ and denote its dominant Newton point by $\nu(\dot x)\in X_\ast(T)_{\Gamma_0}\otimes\mathbb Q$. Consider a reduction tree $\mathcal T$ and the unique path $\underline p$ in $\mathcal T$ which only consists of type II edges, i.e.\ $\ell_I(\underline p)=0$. Then $x$ is a positive Coxeter element if and only if the following conditions are both satisfied:
\begin{enumerate}[(i)]
\item The element $w\in W$ is $\sigma$-conjugate to a partial $\sigma$-Coxeter element in $W$.%_{I(\nu(\dot x))}$.
\item There exists some $v\in \LP(x)$ with
\begin{align*}
\dim X_{\underline p} = \frac 12\left(\ell(x)+\ell(v^{-1}\sigma(wv)) - \langle \nu(\dot x),2\rho\rangle - \defect([\dot x])\right).
\end{align*}
\end{enumerate}
If both conditions are satisfied and $v$ is as in (ii), then $(x,v)$ is a positive Coxeter pair.
\end{proposition}
It is clear that part (i) is a necessary condition for $x$ to be of positive Coxeter type.

The path $\underline p$ ends in a minimal length element of $\mathcal O = \{y^{-1}x\sigma(y)\mid y\in\widetilde W\}$. In particular, $X_{\underline p}\subseteq X_x(\dot x)$.

Recall He's virtual dimension formula: For the one element $v\in \LP(x)$ which is given by the shortest element in $W$ such that $v^{-1}\mu$ is dominant, we set \cite[Section~10.1]{He2014}
\begin{align*}
d_x(b) = \frac 12\left(\ell(x)+\ell(v^{-1}\sigma(wv)) - \langle \nu(b),2\rho\rangle - \defect(b)\right),\qquad [b]\in B(G).
\end{align*}
It is proved by He that $\dim X_x(b)\leq d_x(b)$ always holds true \cite[Theorem~2.30]{He2015}, and that equality holds if $x$ is in a shrunken Weyl chamber and $[b]$ is \enquote{sufficiently small}, cf.\ \cite[Theorem~12.1]{He2014} and \cite[Theorem~6.1]{He2021_cordial}.

However, the notion of virtual dimension behaves poorly outside of the shrunken Weyl chambers, similarly to the aforementioned shortcomings of He--Nie--Yu's notion of finite Coxeter part. This led the first named author to conjecture that
\begin{align*}
\dim X_x(b)\leq \min_{v\in \LP(x)}\frac 12\left(\ell(x)+\ell(v^{-1}\sigma(wv)) - \langle \nu(b),2\rho\rangle - \defect(b)\right),~\forall [b]\in B(G)_x.
\end{align*} The right--hand side of this inequality is better behaved with respect e.g.\ to affine Dynkin diagram automorphisms. So condition (ii) means that $\dim X_{\underline p}$ agrees with this conjectural upper bound for $\dim X_x(\dot x)$. If one believes this conjecture, it can be summarized that condition (ii) means that the two inequalities
\begin{align*}
\dim X_{\underline p} \leq \dim X_x(\dot x)\leq \min_{v\in \LP(x)}\frac 12\left(\ell(x)+\ell(v^{-1}\sigma(wv)) - \langle \nu(\dot x),2\rho\rangle - \defect([\dot x])\right)
\end{align*}
are both equalities. Observe that the first inequality is actually an equality e.g.\ if $\mathcal T$ has only one reduction path for $[\dot x]$ (which is a necessary condition for positive Coxeter type). Moreover, the second inequality is known to be an equality if $x$ lies in a shrunken Weyl chamber and $\nu(\dot x)$ is basic \cite[Theorem~2.31]{He2015}.
\begin{proof}[Proof of Proposition~\ref{prop:pctCharacterization}]
We saw earlier that condition (i) is necessary for $x$ to be of positive Coxeter type. So let us assume (i) holds.

Let $e\in \widetilde W$ denote the endpoint of $p$. Then $e$ is a minimal length element in its $\sigma$-conjugacy class, and $\cl(e)\in W$ is $\sigma$-conjugate to a partial $\sigma$-Coxeter element. We know that $\nu(\dot x) = \nu(\dot e)$.

From Theorem~\ref{thm:coxeterConjugacyClasses}, we know that $e$ is of positive Coxeter type. Let $v'\in W$ such that $(e,v')$ is a positive Coxeter pair.% Since $\cl(e)$ is $\sigma$-elliptic, the support of $(e,v')$ must be $\Delta$. 

From the construction of $X_{\underline p}$, we see that
\begin{align*}
\dim X_{\underline p} = \ell_I(p) + \ell_{II}(p) +\dim X_e(\dot x).
\end{align*}
Since $\ell_I(p)=0$, we get $\ell_{II}(p) = \frac 12(\ell(x)-\ell(e))$. Moreover, by Theorem~\ref{thm:fctConsequences}, we have
\begin{align*}
\dim X_e(\dot x) = \frac 12\left(\ell(e) + \#(J/\sigma) - \langle \nu(\dot x),2\rho\rangle - \defect([\dot x])\right),
\end{align*}
where $J$ is the support of $(e,v')$. Observe that $\#(J/\sigma) = \ell_{R,\sigma}(\cl(e)) = \ell_{R,\sigma}(w)$.
Thus we conclude
\begin{align*}
\dim X_{\underline p} &= \frac 12\left(\ell(x)+\ell_{R,\sigma}(w) - \langle \nu(\dot x),2\rho\rangle - \defect([\dot x])\right).
\end{align*}
We see that condition (ii) holds true for a given element $v\in \LP(x)$ if and only if $\ell(v^{-1}\sigma(wv)) = \ell_{R,\sigma}(w)$. This is equivalent to $v^{-1}\sigma(wv)$ being a partial $\sigma$-Coxeter element. This finishes the proof.
\end{proof}

\subsection{The basic case}\label{sec:basic}
The goal of this section is the following splitting result.
\begin{theorem}\label{thm:basicSplitting}
Assume that $G$ is a split group and that $F = \mathbb F_q\doubleparen t$. Let $(x,v)$ be a pair of positive Coxeter type and $[b]\in B(G)_x$ a basic $\sigma$-conjugacy class. Then the iterated fibration from Corollary~\ref{cor:geo} splits for irreducible components.

Explicitly, choose any minimal length element $e$ in the $\sigma$-conjugacy class of $x$ inside $\widetilde W$, and denote by $\ell_{II}(b) = \frac 12\left(\ell(x) - \ell_{R,\sigma}(\cl(x)) +\defect(b)\right)$ the type II length of a reduction path for $b$. If $C_x\subseteq X_x(b)$ and $C_e\subseteq X_e(b)$ denote any irreducible components of these two affine Deligne--Lusztig varieties, then we have
\begin{align*}
C_x\cong \mathbb A^{\ell_{II}(b)} \times C_e.
\end{align*}
Here, we write $\mathbb A^n$ for the $n$-dimensional affine space over $\overline{\mathbb F_q}$. The component $C_e$ is an irreducible component of a classical Deligne--Lusztig variety of Coxeter type as detailed in Corollary~\ref{cor:geo}.
\end{theorem}
We expect that Theorem~\ref{thm:basicSplitting} holds for all $G$ and $F$, and that an analogous result holds for arbitrary $[b]\in B(G)_x$ (with additional direct factors isomorphic to $\mathbb G_m$ as counted in Theorem~\ref{thm:fctConsequences}). We even expect that our proof method can be followed in general, though one would have to revisit the proofs of the following two essential ingredients as well.
\begin{theorem}[{\cite[Theorem~11.3.1]{Goertz2010}}]\label{thm:ghkrDimension}
Assume that $G$ is a split group and that $F = \mathbb F_q\doubleparen t$. Let $x,y,b\in\widetilde W$ such that $\ell(b)=0$. Then the intersection $X_x(b)\cap IyI\subseteq G(\breve F)/I$ is non-empty if and only if $IxI\cap y^{-1}IbyI\neq\emptyset$.

If this is the case, the two subsets of the affine flag variety
\begin{align*}
X_x(b)\cap IyI/I,\qquad (IxI/I)\cap (y^{-1} IbyI/I)
\end{align*}
have the same dimension.\rightqed
\end{theorem}
\begin{theorem}[{\cite[Section~3]{Faltings2003}}]\label{thm:faltings}
Assume that $G$ is a split group and that $F =\mathbb F_q\doubleparen t$. There is a notion of \enquote{big open cell} $\mathcal N^-\subset G(\breve F)$, which coincides with the subgroup of $G(\breve F)$ generated by all negative affine root subgroups. Moreover, it satisfies the following two additional properties.
\begin{enumerate}[(a)]
\item The map $\mathcal N^-\rightarrow \mathcal N^- I/I,~ g\mapsto gI$ is an isomorphism.
\item For any $x\in\widetilde W$, we have $IxI\subseteq x\mathcal N^- I$.\rightqed
\end{enumerate}
\end{theorem}
The relevance of Falting's result with respect to the Deligne--Lusztig reduction method discussed in Section~\ref{sec:dlreduction} is explained as follows.
\begin{corollary}\label{cor:faltingsForAdlv}
Let $x\in\widetilde W, b\in G(\breve F)$ and $s\in\mathbb S_{\af}$ such that $\ell(sx\sigma(s)) = \ell(x)-2$. Write $X_x(b) = X_1\sqcup X_2$ and $X_2\xrightarrow{\phi} X_{sx\sigma(s)}(b)$ as in Proposition~\ref{DL-reduction}. For any $y\in\widetilde W$, we have an isomorphism \begin{align*}
\phi^{-1}(X_{sx\sigma(s)}(b)\cap IyI)~\cong~ \mathbb A^1\times (X_{sx\sigma(s)}(b)\cap IyI)
\end{align*}
that is compatible with the natural projection maps to $(X_{sx\sigma(s)}(b)\cap IyI)$ on either side.
\end{corollary}
\begin{proof}
Recall that
\begin{align*}
X_x(b) = \{gI\in G(\breve F)/I\mid g^{-1}b\sigma(g)\in IxI\}.
\end{align*}
Following the proof of \cite[Corollary~2.5.3]{Goertz2010b}, we see that $X_2\subseteq X_x(b)$ is given by those elements $gI\in X_x(b)$ such that there exists $g_1I\in gIsI$ with $g_1^{-1} b\sigma(g_1)\in Isx\sigma(s)I$. They also prove that, if such an element $g_1I$ exists, it is uniquely determined by $gI$. The map $\phi : X_2\rightarrow X_{sx\sigma(s)}(b)$ is hence given by
\begin{align*}
\{gI\in X_x(b)\mid \exists! g_1I\in gIsI\cap X_{sx\sigma(s)}(b)\}\rightarrow X_{sx\sigma(s)}(b),\qquad gI\mapsto g_1I.
\end{align*}
Consider the restriction of this projection to $X_{sx\sigma(s)}(b)\cap IyI$. By Theorem~\ref{thm:faltings}, we get a continuous and injective map
\begin{align*}
\psi: IyI/I\hookrightarrow y\mathcal N^- I/I\simeq  y\mathcal N^-,
\end{align*}
such that $g_1I = \psi(g_1I)I$ for all $g_1I\in IyI/I$.

Choose an isomorphism $\mathbb A^1\cong IsI/I$. We then define the promised isomorphism via the two mutually inverse maps
\begin{align*}
(X_{sx\sigma(s)}(b)\cap IyI)\times (IsI/I)\rightarrow& \phi^{-1}(X_{sx\sigma(s)}(b)\cap IyI),\\
(g_1I, g_2I)\mapsto& \psi(g_1I) g_2I
\end{align*}
and
\begin{align*}
\phi^{-1}(X_{sx\sigma(s)}(b)\cap IyI)\rightarrow& (X_{sx\sigma(s)}(b)\cap IyI)\times (IsI/I) ,\\
gI\mapsto &(\phi(gI),~\psi(\phi(gI))^{-1} gI).
\end{align*}
This finishes the proof.
\end{proof}

In order to apply the above result of Görtz--Haines--Kottwitz--Reuman, we need the following simple observation.
\begin{lemma}\label{lem:iwahoriDoubleCosetEstimate}
Let $x,y,z\in\widetilde W$ such that $IxI\cap yIzI\neq\emptyset$. Then
\begin{align*}
\dim (IxI/I~\cap~yIzI/I)\geq \frac 12\left(\ell(x)-\ell(y)+\ell(z)+\ell_R(\cl(x^{-1}yz))\right).
\end{align*}
\end{lemma}
\begin{proof}
Induction on $\ell(z)$. In case $\ell(z)=0$, we get $x = yz$, $\ell(x)=\ell(y)$ and $\dim (yIzI/I)=0$. The claim follows easily.

Let now $\ell(z)>0$ and $a\in \Delta_{\af}$ be a simple affine root with $\ell(r_a z)=\ell(z)-1$.

First consider the case where $y a \in \Phi_\af^+$.
Note that the affine root subgroup $U_a$ acts freely on $IzI/I$ by left multiplication. A fundamental domain for this action is given by $(I\cap r_a I r_a) zI/I$. Note that $r_a I r_a = (I\cap r_a I r_a)\cdot U_{-a}$, with $U_{-a} = r_a U_a r_a$ being the affine root subgroup associated to $-a$. Since $U_{-a}\subseteq zIz^{-1}$, we conclude that the fundamental domain is
\begin{align*}
(I\cap r_a I r_a) zI/I = r_a I r_a zI/I.
\end{align*}
Multiplying by $y$, we see that $U_{ya}$ acts freely on $yIzI$ and a fundamental domain is given by $y r_a Ir_a zI$. Since $U_{ya}\subseteq I$, we see that the left $U_{ya}$-action preserves $IxI$. Thus, $U_{ya}$ acts freely on $IxI\cap y IzI$ with a fundamental domain given by $IxI\cap y r_a I r_a zI$.

Using induction, we get
\begin{align*}
\dim(IxI/I~\cap~yIzI)&\geq \dim U_{ya}+\dim(IxI/I~\cap~yr_a Ir_az I)
\\&\underset{\text{ind}}\geq 1+ \frac 12\left(\ell(x) - \ell(yr_a) + \ell(r_a z) + \ell_R(\cl(x^{-1} yr_a r_a z))\right)
\\&=1+\frac 12\left(\ell(x) - (\ell(y)+1) + (\ell(z)-1)
+ \ell_R(\cl(x^{-1} yz))\right),
\end{align*}
which is the desired claim.

Consider now the case where $y a \in \Phi_\af^-$. Then we get
\begin{align*}
IyIzI = Iyr_a I r_a z I~\cup~IyIr_a z I,
\end{align*}
so at least one of the two intersections
\begin{align*}
IxI~\cap~yr_a Ir_az I,\qquad IxI~\cap~yIr_a zI
\end{align*}
is nonempty.

First consider the case where $IxI\cap yr_a I r_a zI\neq\emptyset$. As above, we get $yr_a I r_a z I\subseteq yIzI$ (this time, the left $U_{ya}$-action might not preserve $IxI$). We conclude
\begin{align*}
\dim (IxI/I~\cap~yIzI)&\geq \dim (IxI/I~\cap~yr_a I r_a zI/I)
\\&\underset{\text{ind.}}\geq
\frac 12\left(\ell(x) - \ell(yr_a) + \ell(r_a z) + \ell_R(\cl(x^{-1} yr_a r_a z))\right)
\\&=\frac 12\left(\ell(x) - (\ell(y)-1) + (\ell(z)-1)
+ \ell_R(\cl(x^{-1} yz))\right),
\end{align*}
which is the desired claim.

So it remains to study the case $IxI\cap y Ir_a zI\neq\emptyset$. For any $1\neq g\in U_{-a}$, we have $g\in Ir_a I$ and hence $gIr_a zI\in I zI$. So $(U_{-a}\setminus\{1\}) Ir_a zI$ is a subset of $IzI$, and for each $h\in (U_{-a}\setminus\{1\})Ir_a zI$, there is a unique $g\in U_{-a}\setminus\{1\}$ with $gh\in Ir_a zI$.

Multiplying by $y$, we see that $(U_{-ya}\setminus\{1\})yIr_azI$ is contained in $yIzI$ and for each $h\in (U_{-ya}\setminus\{1\})yIr_azI$, there exists a unique $g\in U_{-ya}\setminus\{1\}$ with $gh\in yIr_azI$. Note that $g\in I$, so that left multiplication by $g$ preserves $IxI$. As above, we conclude
\begin{align*}
\dim (IxI/I~\cap~yIzI/I)&\geq \dim (IxI/I~\cap~(U_{-ya}\setminus\{1\})yIr_azI)
\\&\geq 1+\dim (IxI/I~\cap~yIr_azI)
\\&\underset{\text{ind.}}\geq 1+\frac 12\left(\ell(x)-\ell(y)+\ell(r_a z)+\ell_R(\cl(x^{-1}yr_a z))\right)
\\&=\frac 12\left(\ell(x)-\ell(y)+\ell(z)+1+\ell_R(\cl(x^{-1} y z r_{z^{-1} a}))\right).
\end{align*}
Since multiplying by a reflection changes reflection length by at most once, we see $1+\ell_R(\cl(x^{-1} y z r_{z^{-1} a}))\geq \ell_R(\cl(x^{-1} yz))$. This finishes the induction and the proof.
\begin{comment}Observe that $r_a I r_a\subseteq I\cup Ir_a I$, so that
\begin{align*}
yIzI &= yr_a (r_a I r_a) r_a zI = yr_a\Bigl[ (r_a I r_a\cap I)\cup(r_a I r_a\cap I r_a I)\Bigr] r_a zI
\\&\subseteq yr_a (I\cap \prescript{r_a}{}{}I)r_a zI~\cup~yr_a I r_a I r_a z I
\\&=yr_a (I\cap \prescript{r_a}{}{}I)r_a zI~\cup~yr_a I z I.
\end{align*}
Hence we get
\begin{align*}
\dim (IxI \cap yIzI/I)\geq \max(\dim (IxI\cap yr_a(I\cap \prescript{r_a}{}{} I) r_a zI/I), \dim (IxI\cap yr_a IzI/I)).
\end{align*}
For the first term: Observe that the affine root subgroup $U_{-a}\subset G(\breve F)$ acts freely on $Ir_a zI/I$ by left multiplication, with a fundamental domain given by $(I\cap \prescript{r_a}{}{}I )r_a z I/I$. Thus $U_{-y^{-1} a}$ acts freely on $IxI\cap yr_a(I\cap I r_a zI/I)$ with a fundamental domain given by $IxI\cap yr_a(I\cap \prescript{r_a}{}{} I) r_a zI/I$. One estimates the dimension as in the above calculation.

For the second term, we use the inductive assumption to see
\begin{align*}
\dim (IxI\cap yr_a IzI/I)&\geq \frac 12\left(\ell(x)-\ell(r_a y) + \ell(z) + \ell_R(\cl(xyr_a z))\right)
\\&=\frac 12\left(\ell(x)+1-\ell(y)+\ell(z)+\ell_R(\cl(xyz) s_{\cl(z^{-1} a)})\right).
\intertext{
Since right multiplication by the reflection $s_{\cl(z^{-1} a)}$ changes the reflection length by at most one, we see}
\ldots&\geq \frac 12\left(\ell(x)-\ell(y)+\ell(z)+\ell_R(\cl(xyz))\right).
\end{align*}
This finishes the induction and the proof.
\end{comment}
\end{proof}
\begin{proposition}\label{prop:coxeterAdlvInSchubertCell}
Assume that $G$ is a split group and that $F = \mathbb F_q\doubleparen t$. Let $x$ be an element of positive Coxeter type, $b\in\widetilde W$ be an element of length zero and $y \in W_a$ a Bruhat minimal element such that $X_x(b)\cap IyI/I\neq\emptyset$. Then $X_x(b)\cap IyI$ contains an irreducible component of $X_x(b)$ and $X_x(b)\subseteq \mathbb J_b IyI/I$.
\end{proposition}
\begin{proof}
By Theorem~\ref{thm:ghkrDimension} and Lemma~\ref{lem:iwahoriDoubleCosetEstimate}, we get
\begin{align*}
\dim (X_x(b)\cap IyI) &= \dim\Bigl((IxI/I)\cap (y^{-1} IbyI/I)\Bigr)  \\&\geq \frac 12\left(\ell(x)-\ell(y)+\ell(by) + \ell_R(\cl(xy^{-1} b y))\right).
\end{align*}
Observe that $\ell(y)=\ell(by)$ since $\ell(b)=0$ by assumption. Moreover, we can estimate
\begin{align*}
\ell_R(\cl(xy^{-1} by)) \geq \ell_R(\cl(x)) - \ell_R(\cl(y^{-1} by)) = \ell_R(\cl(x))-\ell_R(b) = \ell_R(\cl(x))-\defect(b).
\end{align*}
We conclude
\begin{align*}
\dim (X_x(b)\cap IyI)\geq \frac 12\left(\ell(x)+\ell_R(\cl(x))-\defect(b)\right).
\end{align*}
By Theorem~\ref{thm:fctConsequences}, the right--hand side agrees with $\dim X_x(b)$.

Note that $X_x(b)\cap \overline{IyI/I}$ is a closed subset of $X_x(b)$. Here, $\overline{IyI/I}$ is the closure of $IyI$ inside the affine flag variety, which coincides with the union of all $Iy'I$ for $y'\leq y$. By the choice of $y$, we get $X_x(b)\cap Iy'I=\emptyset$ for all $y'<y$. Hence $X_x(b)\cap IyI$ is a closed subset of $X_x(b)$. The above calculation shows that it contains an irreducible component $C$ of $X_x(b)$. By Theorem~\ref{thm:fctConsequences}, we get $X_x(b)=\mathbb J_b C\subseteq \mathbb J_b IyI$.
\end{proof}
\begin{proof}[Proof of Theorem~\ref{thm:basicSplitting}]
By Theorem~\ref{thm:fctConsequences}, any two irreducible components of $X_x(b)$ are isomorphic, and the same statement holds for $X_e(b)$. So the choice of $C_x$ resp.\ $C_e$ does not affect the result. Moreover, this theorem shows that the Deligne--Lusztig reduction method yields, inside each fixed reduction tree, a unique reduction path for $[b]\in B(G)$ consisting only of type II edges.

We do induction on the length of such a reduction path. If $x$ has minimal length in its $\sigma$-conjugacy class, then $X_x(b)\cong X_e(b)$ since $x$ and $e$ are strongly conjugate \cite[Theorem~A]{He2014b}. We get $\ell_{II}(b)=0$, so $C_x\cong C_e$ is indeed the claim we had to prove.

In the inductive step, we apply the Deligne--Lusztig reduction method as in the proof of Theorem~\ref{thm:fctConsequences}. So we have to consider elements $x' = sx\sigma(s)$ with $s\in\mathbb S_{\af}$ such that $\ell(x')\leq \ell(x)$ and the claim being proved for $x'$ already.

If $\ell(x') = \ell(x)$, then $X_x(b)\cong X_{x'}(b)$ and also the value of $\ell_{II}(b)$ does not change if we pass from $x$ to $x'$. Assuming the claim being proved for $x'$ by induction, we see that it follows for $x$ as well.

Consider now the case $\ell(x') = \ell(x)-2$. By Proposition~\ref{prop:coxeterAdlvInSchubertCell}, we find some $y\in\widetilde W$ with $X_{x'}(b)\cap IyI$ containing an irreducible component $C_{x'}$ of $X_{x'}(b)$. Using Corollary~\ref{cor:faltingsForAdlv}, we find a closed irreducible subset $C\subseteq X_x(b)$ with $C\cong \mathbb A^1\times C_{x'}$. Comparing dimensions as detailed in Theorem~\ref{thm:fctConsequences}, we see that $C$ is an irreducible component of $X_x(b)$. By induction, we get $C_{x'}\cong \mathbb A^{\ell_{II}(b)-1}\times C_e$. Hence the claim for $C_x = C$ follows. This finishes the induction and the proof.
\end{proof}

\printbibliography
\end{document}